\documentclass[11pt]{amsart}
\usepackage{amssymb,amsfonts,latexsym}
\usepackage{graphics,verbatim}
\usepackage{graphicx}

\newtheorem{theorem}{Theorem}[section]
\newtheorem{lemma}[theorem]{Lemma}
\newtheorem{proposition}[theorem]{Proposition}
\newtheorem{corollary}[theorem]{Corollary}
\newtheorem{conj}[theorem]{Conjecture}

\theoremstyle{definition}
\newtheorem{definition}[theorem]{Definition}
\newtheorem{example}[theorem]{Example}

\theoremstyle{remark}

\numberwithin{equation}{section}

\newcommand{\R}{{\mathbb R}}
\newcommand{\Z}{{\mathbb Z}}

\newcommand{\N}{{\mathbb N}}

\newcommand{\E}{{\mathbb E}}

\newcommand{\e}{\varepsilon}

\usepackage[usenames]{color} 

\begin{document}

\title  {Hausdorff and Fourier dimension of graph of continuous additive processes}

\author{Dexter Dysthe}
\address{Department of Mathematics, San Francisco State University,
1600 Holloway Avenue, San Francisco, CA 94132.}
 \email{ddysthe@mail.sfsu.edu, drd2144@columbia.edu}

\author{Chun-Kit Lai}

\address{Department of Mathematics, San Francisco State University,
1600 Holloway Avenue, San Francisco, CA 94132.}
 \email{cklai@sfsu.edu}


\begin{abstract}
An additive process is a stochastic process with independent increments and that is continuous in probability. In this paper, we study the almost sure Hausdorff and Fourier dimension of the graph of continuous additive additive processes with zero mean.  Such processes can be represented as $X_t = B_{V(t)}$ where $B$ is Brownian motion and $V$ is a continuous increasing function. We show that these dimensions depend on the local uniform H\"{o}lder indices. In particular, if $V$ is locally uniformly bi-Lipschitz, then the Hausdorff dimension of the graph will be 3/2. 
We also show that the Fourier dimension almost surely is positive if $V$ admits at least one point with positive lower H\"{o}lder regularity. 

\medskip
 It is also possible to estimate the Hausdorff dimension of the graph through the $L^q$ spectrum of $V$. We will show that if $V$ is generated by a self-similar measure on ${\mathbb R}^{1}$ with convex open set condition, the Hausdorff dimension of the graph can be precisely computed by its $L^q$ spectrum.  An illustrating example of the Cantor Devil Staircase function, the Hausdorff dimension of the graph is $1+\frac12\cdot\frac{\log 2}{\log 3}$.   Moreover, we will show that the graph of the Brownian staircase surprisingly has Fourier dimension zero almost surely. 
\end{abstract}

\maketitle

\section{introduction}

Let $X_t$ be a real-valued stochastic process with continuous sample paths defined on a probability space $(\Omega,{\mathcal F}, {\mathbb P})$.  Determining almost surely different dimensions of random fractals generated by $X_t$ has  been an active research problem in fractal geometry and probability.  In this paper, we are interested in the graph of random processes, particularly additive processes with continuous sample paths. These graphs model many real-life applications such as financial asset prices. 

\medskip

In previous studies of dimensions of  random processes, stationary increments played an important role.  Processes with stationary independent increments and that are continuous in probabillity are called L\'{e}vy processes. Our interest, additive processes, possess independent increments and continuity in probability, but do not necessarily have stationary increments.  (See Section 2 for precise definition).	Clearly, additive processes contain all L\'{e}vy processes.  Classcial results of Blumental-Getoor \cite{BG1960} computed  the Hausdorff dimension of the image of Borel sets for L\'{e}vy processes using certain indices of the growth rate of the L\'{e}vy-Khintchine exponents.   Blumental-Getoor indices were generalized to additive processes by Yang and it was shown that the dimension of the images of additive processes can be computed through these indices \cite{Yang2007}.  In the survey paper by Xiao \cite{Xiao2003}, additive processes were mentioned, however, not much in depth work about the dimensions of their graphs or level sets were carried out since then. 

\medskip

  In studying graphs of random processes, the {\it graph of $X$ on  a Borel set $E\subset \R$} is defined by 
$$
{\mathcal G} (X,E) = \{(t, X_t): t\in E\}.
$$ 
It is well-known that L\'{e}vy processes with continuous sample paths must be Brownian motions. Taylor \cite{Taylor1953} showed that the almost sure Hausdorff dimension of the graph of Brownian motion over an interval is $3/2$. This work was generalized to Gaussian process with stationary increments by Orey \cite{Orey1970} and later to Gaussian random fields by Adler \cite{Adler1977}. In particular, the Hausdorff dimension of the graph of fractional Brownian Motion with Hurst index $H$ is $2-H$. For more basic results about graphs, images and level sets of Brownian motions, readers are referred to Falconer \cite[Chapter 16]{Falconer2013} and M\"{o}rters and Peres \cite{MP2012}. For related developments of the Hausdorff dimension of the graph of L\'{e}vy processes and Gaussian random fields, readers are referred to the survey paper by Xiao \cite{Xiao2013,Xiao2003}.

\medskip

Apart from Hausdorff dimensions, the study of Fourier dimension was also very well-received in recent research. For a Borel measure $\mu$ on $\R^2$, the Fourier transform is defined to be 
$$
\widehat{\mu}(\xi) = \int e^{-2\pi i \xi\cdot x}d\mu(x).
$$
The {\it Fourier dimension} of a Borel set $E$ is defined to be 
$$
\mbox{dim}_F(E) = \sup \{\alpha\ge 0: \exists \ \mbox{finite Borel measure $\mu$ supported on $E$ s.t.} \  |\widehat{\mu}(\xi)|  \lesssim |\xi|^{-\alpha/2}\}.
$$
Through potential theoretic methods, it is known that  Fourier dimension is majorized by the Hausdorff dimension for Borel sets. Sets whose Fourier dimension equals their Hausdorff dimension are called {\it Salem sets}, and a measure that admits Fourier decay is called a {\it Rajchman measure}. Fourier dimension encodes many geometric and arithmetic properties of the underlying set that Hausdorff dimension is blind to. However, determining if a set supports a Rajchman measure or if it is Salem is not a straightforward task. Recently, sets with positive Fourier dimension, or sets supporting a Rachman measure, were determined in the study of spectral theory of hyperbolic geometry \cite{BD2017}, self-affine measures \cite{LS2020}, and Gibbs measures associated with Gauss maps \cite{JS2016}. In particular, sets of well-approximable numbers found in diophantine approximation are known to be the only deterministic class of Salem sets \cite{Hambrook2017}.

\medskip

Salem sets and Rajchman measures are ubiquitous in the theory of random processes. Initiated by Kahane in the 1960s, it has been known that the image of a set under fractional Brownian motion is almost surely Salem \cite{KSbook}. This was generalized to Gaussian random fields by Shieh and Xiao in \cite{SX2006}. In the same paper, Shieh and Xiao asked if the graph or the level sets of fractional Brownian motion (fBM) are almost surely Salem. Mukeru \cite{Mukeru2018} recently showed that the zero sets of  fBM are almost surely Salem. However, Fraser, Orponen and Sahlsten proved that the Fourier dimension of the graph of any function defined on $[0,1]$ is at most 1 \cite{FOS2014}, which in turns shows that graph of fBM is almost surely not Salem. Fraser and Sahlsten  further showed that the Fourier dimension of the graph of Brownian motion is 1 almost surely \cite{FS2018}, leaving the case of fBM as an open problem. 

\medskip

\subsection{Main results.}  Our paper continues the line of research of studying Hausdorff and Fourier dimension of the graph of stochastic processes by focusing on centered continuous additive processes. In determining the dimensions of stationary Gaussian processes or random fields, the regularity of the variance function
$$
\sigma^2(s,t) = \E[|X_s-X_t|^2]
$$
plays an important role. For stationarity, $\sigma^2(s,t)$ depends only on $|s-t|$. For isotropic Gaussian random fields, $\sigma^2(s,t)\asymp g(|s-t|)$ for some positive function $g:\R^+\to\R^+$ was assumed.   These assumptions can be regarded as global properties.   The key feature of our result is the requirement of the local regularity properties of the variance function actually determine the dimensions of continuous additive processes.

\medskip

A very detailed classification of additive processes can be found in It\^{o}'s book \cite{ito2013stochastic}. Indeed, the increments of additive processes with continuous sample paths must be normally distributed. We further assume that the process $X_t$ has zero mean (i.e. $\E[X_t] = 0$ for all $t$) and we say the process is {\it centered}. This process  is a continuous martingale and the classical Dambis-Dubin-Schwartz theorem (see e.g. \cite[p. 174]{KSbook}) shows that it can be represented as a time-changed Brownian motion with deterministic quadratic variation. More precisely,
$$
X_t = B_{V(t)}
$$
where $B_t$ is Brownian motion and $V(t)$ is an increasing function (see Theorem \ref{dds} for an independent proof). Such a process is also sometimes be referred as multifractal processes (see e.g. \cite{Riedi} for its statistical properties).  Here, one can see that $\sigma^2(s,t) = V(t)-V(s)$.  Our estimation of the almost sure Hausdorff and Fourier dimension of the graph of additive processes will be based on this representation and the local regularity of $V$.

\medskip

As one can see, if $V$ is constant over an interval, the sample paths of $X_t$ will be constant. If there is only finitely many subintervals for which $V$ is constant, the dimension determination problem can be reduced to the case that $V$ is strictly increasing on an interval $J$.  The local regularity of a strictly increasing function $V$ will be captured through {\it the upper and lower local uniform H\"{o}lder indices} at a point $t\in J$, denoted respectively by $\alpha^{\ast}(t)$ and $\alpha_{\ast}(t)$, of $V$ which are defined as follows:
\begin{definition} Let $V$ be a strictly increasing function defined on an interval $J$.
The {\it upper local uniform H\"{o}lder index} at a point $t\in J$   is 
$$
\alpha^*(t) = \sup \left\{\alpha\in [0, \infty)\,:\, \lim_{\delta\rightarrow 0}\left( \sup_{u_1, u_2 \in  I(t,\delta),  \ u_1\ne u_2} \frac{|V(u_1) - V(u_2)|}{|u_1 - u_2|^\alpha}\right) = 0 \right\}.
$$
The {\it lower local uniform H\"{o}lder index} at a point $t\in J$ is
$$
 \alpha_*(t) = \inf \left\{\alpha\in [0, \infty)\,:\, \lim_{\delta\rightarrow 0}\left( \sup_{u_1, u_2 \in I(t,\delta),  \ u_1\ne u_2} \frac{|u_1 - u_2|^\alpha}{|V(u_1) - V(u_2)|}\right) = 0 \right\}
$$
Finally, we define
$$
\alpha^{\ast}(J) = \inf_{t\in J} \alpha^{\ast}(t) \  \ \mbox{and} \ \  \alpha_{\ast} (J)= \inf_{t\in J} \alpha_{\ast}(t).
$$
\end{definition}
Notice that  $\alpha^{\ast}(t)\le \alpha_{\ast}(t)$. Hence, taking infimum, ${\alpha}^{\ast}\le{\alpha}_*$. 
The idea of the local H\"{o}lder index was motivated from Orey \cite{Orey1970} who studied stationary Gaussian processes.
The following two theorems summarize our main results on the almost sure estimates on Hausdorff and Fourier dimension of centered continuous additive processes.

\medskip

\begin{theorem}\label{hausmainresult}
Let $(X(t))_{t\in J}$ be a centered continuous additive process which is represented as $X_t = B_{V(t)}$. Suppose that $V(t)$ is strictly increasing on $J$. Then, almost surely
\begin{equation}
\max\left\{2 - \frac{{\alpha}_{\ast}}{2}, 1\right\} \leq \dim_H \ \mathcal{G}(X, J)\leq 2 - \frac{{\alpha}^{\ast}}{2},
\end{equation}
where ${\alpha}_{\ast} = {\alpha}_*(J)$ and ${\alpha}^{\ast} = {\alpha}^{\ast}(J)$.
\end{theorem}

\medskip

\begin{theorem}\label{fouriermainresult}
Let $(X(t))_{t\in J}$ be a centered continuous additive process which is represented as $X_t = B_{V(t)}$. Suppose that $V(t)$ is strictly increasing on $J$. Suppose that $T = V^{-1}$ admits a point whose upper local uniform H\"{o}lder index is positive. Then almost surely has a positive Fourier dimension. 
\end{theorem}

\medskip

The proof of Theorem \ref{fouriermainresult} will be adapted from Fraser and Sahlsten. We will reorganize the presentation of the proof by extracting the main general lemmas required for their argument to work. We hope that these will be helpful for future research. A more quantitative version of Theorem \ref{fouriermainresult} will be proved in Theorem \ref{thm_Fdim}. In the proof, stochastic calculus and It\^{o}'s formula was used in estimating  Fourier decay in the horizontal directions. We here indicate the main difference between our proof and the one in \cite{FS2018}.  For Brownian motion, it was a coincidence that the drift measure and the quadratic variation measure are both Lebesgue measure, which allowed Fraser and Sahlsten \cite{FS2018} to obtain their sharp Fourier dimension result for Brownian motion. In additive processes where the drift measure is still the Lebesgue measure, but the quadratic variation is now $dV$, we will need an adjustment to obtain our decay result.

\medskip

The proof of Theorem \ref{hausmainresult} will be a refined version of  the classical proofs by Taylor and Orey \cite{Taylor1953, Orey1970}. However, from an elementary analysis argument, one can prove that $\alpha^{\ast}\le 1$ and $\alpha_{\ast}\ge 1$. Therefore, equality in Theorem \ref{hausmainresult} can only be  attained when $\alpha^{\ast} = \alpha_{\ast} =1$ and the Hausdorff dimension is 3/2. Nonetheless, we will see that Theorem \ref{hausmainresult}, together with the countable stability of Hausdorff dimension, settles the Hausdorff dimension of most common regular strictly increasing functions. 

\medskip

\begin{corollary} \label{corollaryhaus}
Let $(X(t))_{t\in J}$ be a centered continuous additive process which is represented as $X_t = B_{V(t)}$. Suppose that, except finitely many points in $J$,  $V$ is locally bi-Lipschitz. Then $\dim_H \ \mathcal{G}(X, J) = 3/2$ and $\dim_F \ \mathcal{G}(X, J) \ge 2/3$.
\end{corollary}

\medskip

Theorem \ref{fouriermainresult}  raises the natural question whether every strictly increasing function must possess a point with positive upper local uniform  H\"{o}lder index. Unfortunately, the answer is no. Hata \cite[Theorem 1.3]{Hata} constructed a strictly increasing and absolutely continuous function $H$ on $[0,1]$ such that for all $\alpha\in(0,1]$ and for all subinterval $J$ in $[0,1]$, 
$$
\sup_{x,y\in J, x\ne y} \frac{|H(x)-H(y)|}{|x-y|^{\alpha}} = \infty.
$$
For this function $H$,  $\alpha^{\ast} = 0$. We do not know if for such increasing function, the associated additive processes still admit some Fourier decay.  Moreover, as $H$ is absolutely continuous, we do not even know if Corollary \ref{corollaryhaus} can be generalized to $V$ being absolutely continuous.  

\medskip

 \subsection{A Multifractal result.} The local uniform H\"{o}lder indices resemble, but not exactly the same as, the notion of local dimension of the associated Lebesgue-Stieltjes integral $dV$. We may expect that certain multifractal formalism may hold for these local H\"{o}lder indices and one may obtain a very sharp result in Theorem \ref{hausmainresult} .  For related work about multifractal analysis of graph of functions, see \cite{DJ,Jin2011}. The main tools for multifractal analysis is the {\it $L^q$ spectrum} for a Borel probability measure $\mu$: 
   $$
 \tau_{\mu}(q) =  \liminf_{r\to0} \frac{1}{\log r} \log \left( \sup_{{\mathcal B}} \sum_{B\in {\mathcal B}} (\mu(B))^q \right)
 $$
  where supremum is taken over all families of disjoint closed balls ${\mathcal B}$ where each $B\in{\mathcal B}$ has the center inside the support of $\mu$.   We say that $\mu$ obeys the {\it multifractal formalism} if  the Hausdorff dimension of the level set of local dimensions
  $$
K_{\alpha} = \left\{ x: \lim_{r\to 0} \frac{\log \mu (B(x,r))}{\log r} = \alpha\right\}
$$
can be recovered by the Legendre transform of the $L^q$ spectrum as follows:
$$
\mbox{dim}_H K_{\alpha} = \tau_V^{\ast}(\alpha)= \inf_{q\in{\mathbb R}}\{q\alpha-\tau_{\mu}(q)\}.
$$ 
Multifractal formalism over all ${\mathbb R}$ or certain specified ranges has been a central study to fractal geometry community verified for large classes of fractal measures. In particular, for self-similar measures with open set condition and some self-similar measures with exact overlaps (see \cite{FengLau}, \cite{BF2021}, \cite{S2019} and the reference therein). In our study, we know that $V(t) = \mu [0,t]$ generates an increasing function and it can be used generate additive processes and we can study the dimensions of its graph. 

\medskip

   In a paper \cite{Jin2011} suggested by the referee, Jin obtained a full multifractal description of graph of $b$-adic Mandelbrot Cascade functions and computed the dimension of its graph in terms of its $L^q$ spectrum. Inspired by this paper, we attempt to study the Hausdorff dimension of the graph of additive processes through its $L^q$ spectrum. The following results are obtained for a type of increasing functions. We notice that these results hold also to fractional Brownian Motion $B^H$ with Hurst index $H\in(0,1)$. The precise definition of self-similar measures will be given in Section \ref{section-self-similar}. Convex open set condition means that we can choose $(0,1)$ to satisfy the well-known open set condition. 
    
   \begin{theorem}\label{theorem_multifractal}
   Let $\mu$ be the self-similar measure on $[0,1]$ generated by an IFS satisfying the convex open set condition and let $V(x) = \mu [0,x]$ be the associated increasing function. Let also $X^H(t) = B^H_{V(t)}$. Then 
$$
{\mbox{\rm dim}}_H \ {\mathcal G}(X^H, [0,1])  =  1- \tau_{\mu}\left(H\right).
$$
\end{theorem}

  \medskip

  We notice that the upper bound by $1- \tau_{\mu}\left(H\right)$ is indeed true without any further assumptions. The challenge is to establish the lower bound. For any Borel probability measures $\mu$, we know that  $\tau_{\mu}(1) = 0$ and    $\tau_{\mu}(0) = -\mbox{$\overline{\rm dim}$}_B (\mbox{supp}(\mu))$, which is the upper box dimension of $\mbox{supp}(\mu)$. As $\tau_{\mu}$ is increasing, we have
 $$
  1 < \mbox{\rm dim}_H \ {\mathcal G}(X^H, [0,1]) <1+ \mbox{$\overline{\rm dim}$}_B (\mbox{supp}(\mu)).
  $$
  
   \medskip
  
  The scenario of Theorem \ref{theorem_multifractal} is a complete opposite to Theorem \ref{hausmainresult} in the sense that  $V$ is indeed constant almost everywhere. A representing example will be $V$ being the famous Cantor devil staircase function and we call such process $B_{V(t)}$ the {\it Brownian Staircase}. Figure 1 shows some of  their sample paths. Usual estimation in Theorem \ref{hausmainresult} appears to fail without additional modifications. To prove Theorem \ref{theorem_multifractal}, we will need to carefully exploit  the self-similar property of the $V$ function. 
  
   \medskip
   
  It is well-known that the $L^q$-spectrum for the standard middle-third Cantor measure is given by 
  $$
  \tau_{\mu}(q) = \frac{\log 2}{\log 3}(q-1), 
  $$
We can easily obtain the Hausdorff dimension of the graph is almost surely $1+\frac{1}{2}\cdot \frac{\log 2}{\log 3}$. The following theorem further computed its Fourier dimension. 
\medskip

\begin{theorem} \label{thm_BStair}
Let $V$ be the middle-third Cantor devil staircase function. Then almost surely the graph of the Brownian staircase $X_t = B_{V(t)}$ has Hausdorff dimension $1+ \frac{\log2}{2\log 3} \sim 1.3155....$.
Moreover, 
$$
\mbox{\rm dim}_F \  {\mathcal G}(X, [0,1]) =0.
$$
 \end{theorem}

\begin{figure}[h]
\begin{center}
\includegraphics[width=10cm]{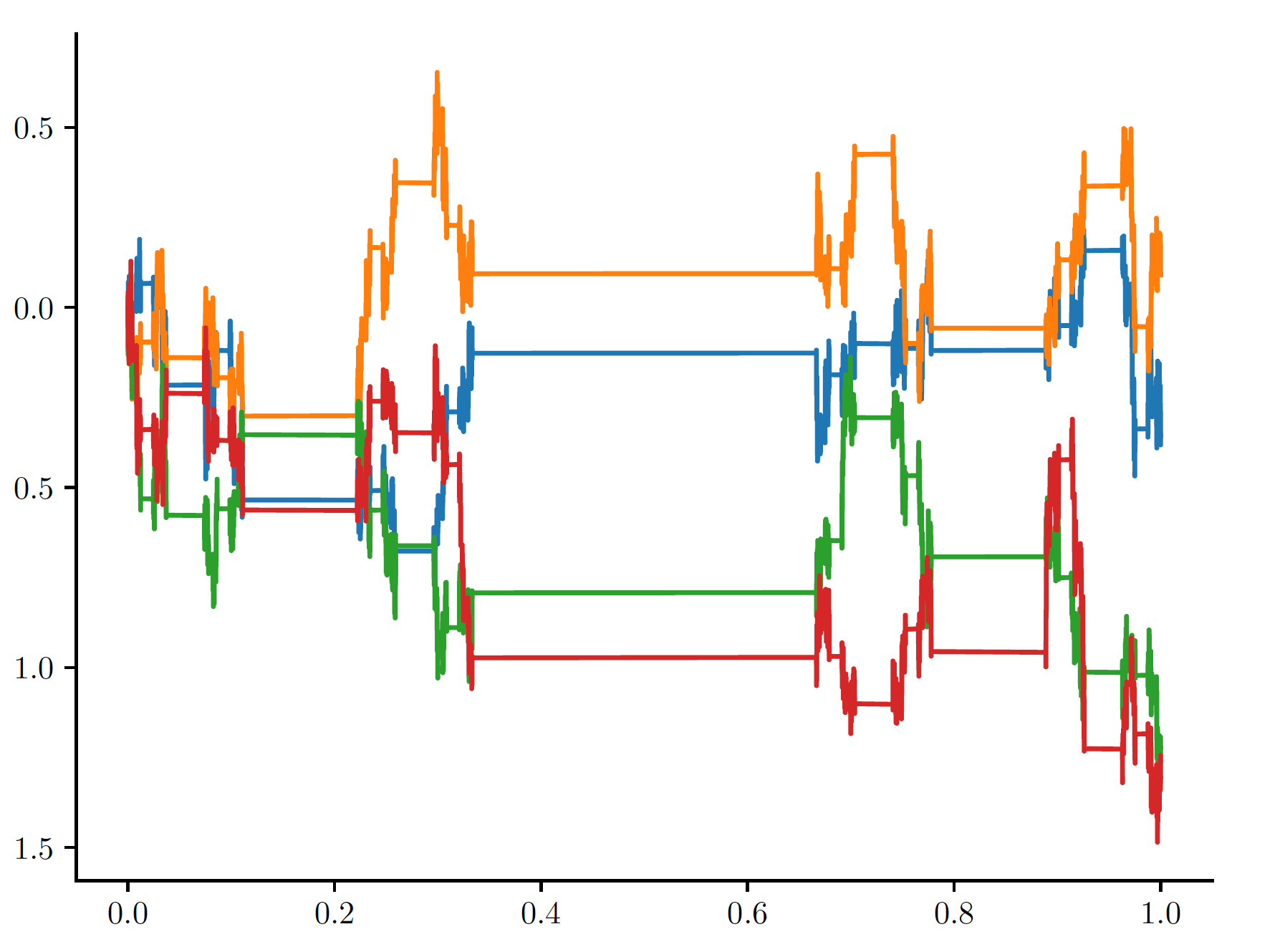}
\end{center}
\caption{Four sample paths of Brownian Staircase}
\end{figure}

It has been widely known that randomization by Brownian motion will likely produce Salem sets or at least measures with polynomial Fourier decay. However,  our results on the Fourier dimension of the Brownian staircase demonstrates that this is not always the case. On the other hand, we will demonstrate another singularly continuous, strictly increasing function $V$ for which $B_{V(t)}$ does have Fourier decay and positive Fourier dimension (See Example \ref{example6.3}).

\medskip

\noindent {\bf Organization of our paper.} Our paper is organized as follows. It can be regarded as  four separate  parts.

\medskip
Part 1 (Section 2-3): we will provide the preliminaries of additive processes and the local H\"{o}lder indices. Then we prove the more elementary results on  Hausdorff dimensions.

\medskip

Part 2 (Section 4-6): We provide some necessary tools to prove our main result for Fourier  decay in section 4  and then in Section 5, Theorem \ref{fouriermainresult} will be proved.
In Section 6, we will illustrate our results with examples and prove Corollary \ref{corollaryhaus}. 

\medskip

Part 3 (Section 7-9): We will  focus on almost constant increasing functions. In Section 7, we introduce the definition of $L^q$ spectrum for general continuous functions and then set up  the terminolgies for  self-similar measures. In Section 8, Theorem \ref{theorem_multifractal}  will be proved. In Section 9, we will shoow that the Brownian Staircase has Fourier dimension zero. 

\medskip

Part 4 (Section 10): We will discuss open questions, remarks  and conjecture.

\medskip

\noindent {\bf Notation of our paper.} We will write $A\lesssim B$ if there exists a constant $C>0$ such that $A\le C B$ where $C$ is independent of any parameters that define the quantites $A$ and $B$. dim$_H$ will denote Hausdorff dimension and dim$_F$ will denote the Fourier dimension. 
 
 \section{Preliminaries} 
  Throughout the paper, $N(\mu,\sigma^2)$ means normal distribution with mean $\mu$ and variance $\sigma^2$. We recall that a stochastic process $(X_t)_{t>0}$ defined on a probability space $(\Omega, {\mathcal F},{\mathbb P})$ has {\it independent increments} if for all $0<t_1<...<t_n$, the random variables $$X_{t_1}-X_0,X_{t_2}-X_{t_1},...,X_{t_n}-X_{t_{n-1}}$$ are independent. Unless otherwise specified, $(\mathcal{F}_t)$  denotes the natural filtration generated by $X_s, s\le t$. A martingale $X_t$ is an integrable stochastic process with 
 $$
 \E[X_t|{\mathcal F}_u] = X_u
 $$
 for all $0\le u<t$. 
 
 \medskip
 
 \subsection{Additive processes} In this paper, $J = [0,S]$, $0<S\le\infty$, denotes the time-interval for which the process is defined.
\begin{definition}  A stochastic process $(X_t)_{t\in J}$ is called an {\it additive process} if $X_t$ has independent increments,  $X_0 = 0$ almost surely, and $X_t$ is continuous in probability  i.e. for all $\epsilon>0$ and $t>0$, 
$$
\lim_{s\to t} {\mathbb P} ( |X_s-X_t|\ge \epsilon) = 0.
$$
 The process is said to be a {\it continuous additive process} if additionally the sample paths of X are continuous almost surely. We say that the additive process is {\it centered} if ${\mathbb E}[X_t] = 0$ for all $t>0$. 
 \end{definition}
 
The definition was taken from the book of  It\^{o} \cite{ito2013stochastic}\footnote{In It\^{o}'s book, the process was called L\'{e}vy process and the usual L\'{e}vy process was called the homogeneous L\'{e}vy process}. It\^{o} actually obtained a much more general decomposition in which continuity in probability is not assumed. However, such process may not even have a finite expectation. Assuming continuity in probability, a theorem of Doob allows one to find a C\`{a}dl\`{a}g modification of the process. Moreover, a remarkable result due to It\^{o} ( See \cite[Section 1.4, Theorem 1]{ito2013stochastic} ) about continuous additive processes is the following.

\medskip

\begin{theorem}[It\^{o}]\label{ito}
Let $(X_t)_{t\in J}$ be a continuous additive process. Then, for any $t,u\in J$
\begin{equation}
X_t - X_u \sim N\Big(m(t,u), V(t,u)\Big)\;.
\end{equation}
where $m, V: J\times J \rightarrow \R$ and $V$ is a non-negative function. 
\end{theorem}

\medskip

There is a convenient property of the variances of increments of continuous additive processes which we prove now. In the remainder of this section we let 
$$
V(t) = V(t,0) \  \mbox{and} \ m(t) = m(t,0).
$$

\medskip

\begin{proposition}\label{Varlemma}
Let $(X_t)_{t\in J}$ be a continuous additive process and let $0\leq u < t$ for $t\in J$. Then 
\begin{enumerate}
\item$V(0) = 0$,
\item $V$ is increasing on $J$,
\item $V$ is a continuous function on $J$,
\item $V(t,u)=V(t)-V(u).$
 \end{enumerate}
 If we further suppose that the process is centered, then $X$ is a martingale with respect to its natural filtration and 
$$
X_t - X_u  \ \sim  \ N(0, V(t) - V(u))
$$
 for all  $0\leq u < t$. In particular, $X_t\sim N(0,V(t))$.
\end{proposition}

\medskip

\begin{proof}
We first show that $V(t,u)=V(t)-V(u)$ holds.  Since $X$ has Gaussian increments, 
$$\E[e^{i\rho X_u}]=e^{i\rho m(u) - \frac{\rho^2V(u)}{2}}, \; u\in J, \; \rho\in\R,$$ and further because $X$ has independent increments,
$$e^{i\rho m(t) - \frac{\rho^2V(t)}{2}}=\E[e^{i\rho X_t}]=\E[e^{i\rho X_u}]\E[e^{i\rho(X_t-X_u)}]=e^{i\rho m(u) - \frac{\rho^2V(u)}{2}}\E[e^{i\rho(X_t-X_u)}], \quad 0\leq u<t.$$
This thus leads to
$$\E[e^{i\rho(X_t-X_u)}]=e^{i\rho(m(t)-m(u)) - \frac{\rho^2(V(t)-V(u))}{2}},$$
which shows that $X_t-X_u$ is normally distributed with mean $m(t)-m(u)$ and variance $V(t)-V(u)$. Hence,  $V(t,u) = V(t) - V(u)$ follows. 

\medskip

We now show the remaining properties of $V$ holds. Indeed, since $X_0 = 0$, we must have $V(0)=0$. Since $V(t,u)\ge 0$, $V$ must be an increasing function. To see that $V$ is a continuous function, let $u\in{\mathbb T}$ be arbitrary. Since $X$ is continuous, almost surely we have 
$$\lim_{t\rightarrow u}\, e^{iX_t} = e^{iX_u},$$
which by Lebesgue's Dominated Convergence Theorem gives 
$$\lim_{t\rightarrow u}\, \E[e^{iX_t}] = \E[e^{iX_u}].$$
Since $X_t\sim N(0, V(t))$ the equality immediately above gives 
$$\lim_{t\rightarrow u}\, e^{-\frac{V(t)}{2}} = e^{-\frac{V(u)}{2}},$$
which shows $V(t)\rightarrow V(u)$ as $t\rightarrow u$ by writing $V(t) = -2\ln (e^{-\frac{V(t)}{2}})$ and using the continuity of logarithm.

\medskip

To prove the martingale property, It follows from a direct calculation as below: for all $u<t$,
$$
\begin{aligned}
\E[X_t \, | \, \mathcal{F}_u] =&  \E[X_t - X_u \, | \, \mathcal{F}_u] + \E[X_u \, | \, \mathcal{F}_u] \\
= &  \E[X_t - X_u ] + X_u  \ (\mbox{by independent increment})\\
=&  X_u  \ (\mbox{as the process is centered}).\\
\end{aligned}
$$
\end{proof}

\medskip

\subsection{Time-changed Brownian Motion.} As we have shown, a centered continuous additive process must be a martingale, and so we can apply the Dambis-Dubin-Schwartz theorem \cite[p.174]{KSbook} to conclude that this process must be a time-changed Brownian motion. To facilitate our discussion, we will need a more precise understanding of this time change.  Given a centered continuous additive process $X_t$ with variance function $V(t)$, we define
$$
T(s) = \inf \{t\ge 0: V(t)>s\}.
$$
if $s\in V(J) = [0,V(S)]$ and $T(s) = +\infty$ if $s>V(S)$. Following immediately from this definition, one can see that $T$ is an increasing function. 
\medskip

\begin{lemma} \label{lemma T}
$T$ satisfies the following
\begin{enumerate}
\item $V(T(s)) = s$, $T(V(t)) = \sup \{s\ge t: V(s) = V(t)\}$.
\item $T$ is strictly increasing.
\item If $V$ is strictly increasing, then $T = V^{-1}$.
\end{enumerate}
\end{lemma}
\medskip

\begin{proof}
(i) can be found in \cite[p.174 and 231]{KSbook}. To see (ii). Suppose that $s_1<s_2$ but $T(s_1) = T(s_2)$ . Using (i), $s_1 = V(T(s_1)) = V(T(s_2)) = s_2$ which is a contradiction. Hence, $T$ is strictly increasing. Using (i), if $V$ is strictly increasing, $T(V(t)) = t$. Hence, $T= V^{-1}$.
\end{proof}

\medskip

We end this section by stating the Dambis-Dubin-Schwartz theorem applied to the centered continuous additive process. 

\begin{theorem}\label{dds}
Let $(X_t)_{t\in J}$ be a centered continuous additive process. Then $B_t = X_{T(t)}$ defines a Brownian motion on $[0,V(S)]$ with $B_0 = 0$. Moreover, $X_t = B_{V(t)}$. 
\end{theorem}

\begin{proof}
Note that $T(0) = \sup \{s: V(s) = 0\}$. If $T(0) = 0$, then $B_0 = 0$ follows. Otherwise, let $s_0 = \sup \{s: V(s) = 0\}$. Then for all $n\in\N$, $V(s_0-1/n) = 0.$ This means that $X_{s_0-1/n} \sim N(0, V(s_0-1/n)) = N(0,0)$ and thus $X_{s_0-1/n} = 0$ almost surely.  Hence, by the continuity of the process, $B_0 = X_{T(0)} = 0.$

\medskip

 If $t_1<t_2<...<t_n$, then $T(t_1)<...<T(t_n)$ as $T$ is strictly increasing. The independent increments property of $X$ immediately shows that $B_t$ also possess independent increments. Finally, $B_t \sim N(0, V(T(t))) = N(0,t)$ by Lemma \ref{lemma T}(i). This shows $B_t$ is a Brownian motion. Using the second part of Lemma \ref{lemma T}(i) we obtain
$$
B_{V(t)} = X_{T(V(t))} = X_{\sup\{s\ge t: V(s) = V(t)\}}.
$$
Again, by the continuity of the sample path of $X$, $X_{\sup\{s\ge t: V(s) = V(t)\}} = X_t$. The proof is complete. 
\end{proof}

\medskip

\subsection{Local H\"{o}lder indices} We now fix a non-empty closed and bounded interval $J$. To calculate the Hausdorff dimension of ${\mathcal G}(X,J)$.  We first let $I(t,\delta)$ to be the open interval of length $\delta$ centered at $t$.  Let $t_0\in J$.  We say that $t_0$ is a {\it non-constant point} of $V$ if for all $\delta>0$, there exists $t\ne t_0$ and $t\in I(t_0,\delta)$ such that $V(t)\ne V(t_0)$. Otherwise, $t_0$ is called a {\it constant point} of $V$. We will denote the set of all non-constant points of $V$ in $J$ by $E_V(J)$. If $J$ is fixed, we will simply write $E_V$. 
\medskip

\begin{lemma}
$E_V$ is a compact set in $\R$. 
\end{lemma}

\begin{proof}
It suffices to prove that $J\setminus E_V$ is open. Indeed, if $t\in [0,1]\setminus E_V$, then there exists $\delta>0$ such that $V$ is constant on $I(t_0,\delta)$ and every point in  $I(t_0,\delta)$ are also constant points, which means that $I(t_0,\delta)$ is a subset of $[0,1]\setminus E_V$. This shows that the set of all non-constant points is open. The proof is complete. 
\end{proof}

We note that $J\setminus E_V$ is open in $\R$. We can write 
\begin{equation}\label{decomposition}
J\setminus E_V= \bigcup_{i=1}^{\infty} J_i
\end{equation}
 as a countable disjoint union of open intervals. We can without loss of generality assume that in each open interval $J_i$,  $V$ must be a constant function. Otherwise, there will be a non-constant point inside $J_i$ and we can further decompose $J_i$ into intervals that $V$ is a constant. Moreover, the end-point of each $J_i$ must be a non-constant point. Otherwise, we can just enlarge the open interval.  
 As 
$$
{\mathcal G}(X,J)= {\mathcal G}(X, E_V) \cup {\mathcal G}(X,J\setminus E_V),
$$
by the countable stability of Hausdorff dimension and $X$ is constant of each $J_i$, we have also the following simple lemma. 

\begin{lemma}\label{graph lemma_0}
$$\mbox{\rm dim}_H({\mathcal G}(X,J)) = \max \{\mbox{\rm dim}_H({\mathcal G}(X, E_V), 1 \}$$
\end{lemma}

\medskip

We can define the upper local uniform index in a slightly more general way for non-constant points $t$.

\begin{definition}The {\it upper local uniform H\"{o}lder index} on the non-constant points $t$ is defined as
$$
\alpha^*(t) = \sup \left\{\alpha\in [0, \infty)\,:\, \lim_{\delta\rightarrow 0}\left( \sup_{u_1, u_2 \in E_V\cap I(t,\delta),  \ u_1\ne u_2} \frac{|V(u_1) - V(u_2)|}{|u_1 - u_2|^\alpha}\right) = 0 \right\}
$$
\medskip
\begin{equation}\label{equpperholder}
{\alpha}^{\ast}(J) = \inf_{t\in J}  \alpha^{\ast}(t).
\end{equation}
\end{definition}

\medskip

\noindent{\bf Remark.} Notice that the analogous definition with non-constant point cannot give us a well-defined $\alpha_{\ast}(t)$. Indeed, if for all $\delta>0$, $I(t,\delta)$ contains infinitely many constant intervals, say one of them is  $(u_1,u_2)$, then $u_1,u_2$ are non-constant points but $V(u_1) = V(u_2)$, so this implies that $\alpha_*(t) = \infty$. The assumption that $V$ is strictly increasing on $J$ avoids such problem. 

\medskip

The following two propositions summarize some important properties of the local uniform H\"{o}lder indices. 

\begin{proposition}\label{propindex} Let $V$ be a continuous and strictly increasing function on the interval $J$ and let $T = V^{-1}$. Denote by $\beta^{\ast}(s)$ the upper local uniform H\"{o}lder index  of $T$ as the point $s$. Then for all $t\in J$. 
$$
\alpha_{\ast}(t) = \frac{1}{\beta^{\ast}(V(t))}.
$$
\end{proposition}

\begin{proof}
Let $\delta>0$ and consider $I(t, \delta)$. Then the image of $I(t,\delta)$ under $V$  is  $(V(t-\delta),V(t+\delta))$. Consider the smallest interval centered at $V(t)$ that contains $(V(t-\delta),V(t+\delta))$ and denote it by $I(V(t),\widetilde{\delta})$. Now, for all $u_1,u_2\in I(t,\delta)$, we write $T(v_1) = u_1, T(v_2) = u_2$ for unique $v_1,v_2$ in  $(V(t-\delta),V(t+\delta))$. We have that
$$
\frac{|u_1-u_2|^{\alpha}}{|V(u_1)-V(u_2)|} = \left(\frac{|T(v_1)-T(v_2)|}{|v_1-v_2|^{1/\alpha}}\right)^{\alpha}\le \left(\sup_{v_1,v_2\in I(V(t),\widetilde{\delta})} \frac{|T(v_1)-T(v_2)|}{|v_1-v_2|^{1/\alpha}}\right)^{\alpha}
$$
Hence, 
$$
\sup_{u_1,u_2\in I(t,\delta)}\frac{|u_1-u_2|^{\alpha}}{|V(u_1)-V(u_2)|} \le \left(\sup_{v_1,v_2\in I(V(t),\widetilde{\delta})} \frac{|T(v_1)-T(v_2)|}{|v_1-v_2|^{1/\alpha}}\right)^{\alpha}
$$
By a direction calculation, $\widetilde{\delta} = \max \{V(t+\delta)-V(t), V(t)-V(t-\delta)\}$. Observe that if $\delta\to 0$, then $\widetilde{\delta}\to 0 $ by continuity. Therefore, if $1/\alpha <\beta^{\ast}(V(t))$, then the right hand side tends to zero and therefore $\alpha \ge \alpha_{\ast}(t)$. Hence, $1/\alpha\le 1/\alpha_{\ast}(t)$ and hence $\beta^{\ast}(V(t)) \le 1/\alpha_{\ast}(t)$. A similar consideration by reversing the role of $V$ and $T$ gives the other side of the inequality. Hence, the proposition follows. 
\end{proof}

\begin{proposition}\label{propindex1}
 If $V$ is strictly increasing on $J$, then $\alpha^{\ast}(t)\le 1$ and $\alpha_{\ast}\ge 1$ for all $t\in J$.
\end{proposition}

\begin{proof}
This proof follows from a standard fact from elementary analysis. We provide a detailed argument here for completeness.  Note that if there exists $t\in J$ such that $\alpha^{\ast}(t)>1$, then let $1<\alpha<\alpha^{\ast}$ in a small neighborhood $I(t,\delta_t)$, we have 
$$
|V(x)-V(y)|\le |x-y|^{\alpha}, \forall x,y\in I(t,\delta_t)
$$
As $\alpha>1$, from elementary analysis, we know that $V'(s) = 0$ for all $s\in I(t,\delta_t)$, which shows $V$ is a constant function on $I(t,\delta_t)$, a contradiction. Hence,  $\alpha^{\ast}(t)\le 1$. The fact that $\alpha_{\ast}\ge 1$ follows from Proposition \ref{propindex}.
\end{proof}

\medskip

\section{Hausdorff dimension by local H\"{o}lder indices}

Let us recall that the Hausdorff dimension of a Borel set $E$ is given by $\mbox{dim}(E) = \sup \{\alpha\ge 0: {\mathcal H}^{\alpha}(E)>0\}$, where
$$
{\mathcal H}^{\alpha}(E) = \lim_{\delta\to 0} \inf \left\{ \sum_{i=1}^{\infty} |U_i|^{\alpha}: E \subset \bigcup_{i=1}^{\infty} U_i,  \ |U_i|\le \delta\right\}.
$$
($|U_i|$ denotes the diameter of the set $U_i$). We will be using frequently the countable stability of Hausdorff dimension (see e.g. \cite[p.49]{Falconer2013}).
$$
\mbox{dim}_H\left(\bigcup_{i=1}^{\infty} E_i\right) = \sup_{i} \mbox{dim}_H(E_i). 
$$
 
 Let $X_t$ be a centered continuous additive process defined on $J$. As in the previous section, we can write $X_t = B_{V(t)}$ where $B$ denotes Brownian motion and $V$ is a continuous increasing function with $V(0) = 0$.   
\medskip

\subsection{Upper bound.} Let $V$ be an increasing continuous function. The upper bound of the Hausdorff dimension will be dependent on the upper local uniform H\"{o}lder index. The following is our main theorem. 

\medskip

\begin{theorem}\label{theorem_upperBound}
 Let $(X_t)_{t\in J}$ be a centered continuous additive process with variance function $V(t)$. Then almost surely
$$
\mbox{\rm dim}_H \ \mathcal{G}(X,J) = 1, \quad \text{if} \;\mbox{\rm dim}_H (E_V) \leq \frac{{\alpha}^{\ast}}{2},
$$
and
$$
\mbox{\rm dim}_H\mathcal{G}(X,J) \leq 1 + \mbox{\rm dim}_H (E_V) - \frac{{\alpha}^*}{2}, \quad \text{if} \; \mbox{\rm dim}_H (E_V) > \frac{{\alpha}^{\ast}}{2},
$$
where ${\alpha}^*= {\alpha}^*(J)$.
\end{theorem}

\medskip

We first need the following lemma, which generalizes a classical result on an upper bound for the Hausdorff dimension of graph of $\alpha$-H\"{o}lder continuous functions over an interval is at most $2-\alpha$ \cite[Ch 11]{Falconer2013}. The generalization is straight-forward, but we present it here for the sake of completeness.

\begin{lemma}\label{graph lemma}
Let $E$ be a Borel set of $\R$ and let $f:E\to\R$ be a H\"{o}lder continuous function with exponent $\alpha$. i.e. there exists $C>0$ such that for all $x,y\in E$, 
$$
|f(x)-f(y)|\le C|x-y|^{\alpha}.
$$
Then 
$$
\mbox{dim}_H ({\mathcal G} (f,E)) \le 1+\mbox{dim}_H (E) - \alpha. 
$$
\end{lemma}

\begin{proof}
Let $s>\mbox{dim}(E)$ and let $\epsilon, \delta>0$. Then ${\mathcal H}^{s}_{\delta}(E) <\infty$ and there exists a countable cover $\{U_i\}$ of $E$ such that 
$$
 \sum_{i=1}^{\infty} |U_i|^s\le {\mathcal H}^{s}_{\delta}(E)  +\epsilon
$$
with $|U_i|\le \delta$. On each $U_i$, by the H\"{o}lder condition, the graph ${\mathcal G}(f,U_i)$ is covered by $U_i\times I_i$ where $I_i$ is an interval such that $|I_i| \le C|U_i|^{\alpha}$. Hence,  the diameter of ${\mathcal G}(f,U_i)$ is at most $C' |U_i|^{\alpha}$ for some $C'>0$. We now divide $I_i$ into smaller intervals of length $|U_i|$; then ${\mathcal G}(f,U_i)$ can be covered by $C' |U_i|^{\alpha-1}$ number of covers of  diameter at most $\sqrt{2}|U_i|$. Summing over all such covers to the power $t = 1+s-\alpha$, 
$$
{\mathcal H}_{\sqrt{2}\delta}^{t}({\mathcal G}(f,E)) \le \sqrt{2^t}C'\cdot \sum_{i=1}^{\infty}  |U_i|^{\alpha-1}  |U_i|^t \le \sqrt{2^t}C' \cdot({\mathcal H}^{s}_{\delta}(E)  +\epsilon) 
$$ 
Hence, 
$$
\mbox{dim}_H ({\mathcal G}(f,E)) \le  1+s-\alpha.
$$ 
Letting $s$ approach $\mbox{dim}_H(E)$, our desired result follows. 
\end{proof}

\medskip

\noindent{\it Proof of Theorem \ref{theorem_upperBound}.} If ${\alpha}^{\ast} = 0$ then we may apply monotonicty and the product formula \cite[Ch. 7]{Falconer2013} for Hausdorff dimension to obtain
$$\dim_H\mathcal{G}(X, E_V) \leq \dim_H\left(\,E_V\times \left[\min_{t\in E_V} X(t),\, \max_{t\in E_V} X(t)\right]\right) = 1 + \dim_H E_V.$$
We now assume ${\alpha}^{\ast}>0$ and let $0<\eta<1$. Then, for all $t\in E_V$ we have $\alpha^{\ast}(t)>0$ and consequently there exists $\delta_t>0$ such that
\begin{equation}
|V(u_1)-V(u_2)|<|u_1 - u_2|^{(1-\eta)\alpha^*(t)}, \quad \forall u_1, u_2 \in E_V\cap I(t, \delta_t).
\end{equation}
Now, since $X_t = B_{V(t)}$ and $(B_t)$ has almost surely $\alpha$-H\"{o}lder sample paths for all $0<\alpha<\frac{1}{2}$, we have that $(X_t)_{t\in E_V\cap I(t,\delta)}$ has almost surely $\alpha$-H\"{o}lder sample paths for all $0<\alpha<(1-\eta)\frac{\alpha^*(t)}{2}$. By using the compactness of $E_V$ we may then obtain $\{t_1, \ldots ,t_N\}\subset E_V$ for which 
$$E_V = \bigcup_{k=1}^{N} E_V\cap I(t_k, \delta_{t_k}).$$
Thus, applying countable stability and monotonicity yields
\begin{align*}
\dim_H\mathcal{G}(X,E_V) &= \max_{k\in\{1,\ldots, N\}} \dim_H\mathcal{G}(X,E_V\cap I(t_k, \delta_{t_k})) \\
&\leq \max_{k\in\{1,\ldots, N\}} \left(1+\dim_H E_V - (1-\eta)\frac{\alpha^*(t_k)}{2}\right) \ (\mbox{by Lemma \ref{graph lemma}}) \\ 
&\leq 1+\dim_H E_V - (1-\eta)\frac{\alpha^*}{2},
\end{align*}
and since $\eta>0$ is arbitrary we have 
$$
\dim_H\mathcal{G}(X, E_V) \leq 1+\dim_H E_V - \frac{\alpha^*}{2}.
$$
Combining with Lemma \ref{graph lemma_0}, the conclusion of this theorem follows. 
\qquad$\Box$

\medskip

\subsection{Lower bounds.} We now aim to prove  Theorem \ref{hausmainresult}. As $E_V = J$ for $V$ strictly increasing, the upper bound follows from Theorem \ref{theorem_upperBound}. We just need to establish the lower bound. Denote by ${\mathcal M}(E)$ the set of all finite Borel measures that are compactly supported on $E$.  We recall that for a measure $\mu\in\mathcal{M}(E)$, the $s$-energy of $\mu$ is defined as
$$
I_{s}(\mu) = \int\int \frac{1}{|x-y|^s}d\mu(x)d\mu(y).
$$
Moreover, 
$$
\mbox{\rm dim}_H E = \sup\{s \, : \, \exists \, \mu\in\mathcal{M}(E) \; \text{s.t.} \; I_s(\mu) < \infty\}.
$$
\medskip

\begin{lemma}\label{boundonexpec}
There exists $C>0$ such that for all $t,u \in J$
\begin{equation}\label{eq_bound_expect}
\mathbb{E}\left[\frac{1}{\left(|t-u|^2 + |X_t - X_u|^2\right)^{\frac{s}{2}}} \right] \leq C \frac{|t-u|^{-s+1}}{\sqrt{|V(t)-V(u)|}}.
\end{equation}
\end{lemma}

\begin{proof}
Since $X_t - X_u \sim N(0,\, V(t)-V(u))$ we have 
\begin{align*}
&\mathbb{E}\left[\frac{1}{\left(|t-u|^2 + |X_t - X_u|^2\right)^{\frac{s}{2}}} \right] \\[1em]
&= \frac{1}{\sqrt{2\pi (V(t)-V(u))}} \int_{0}^{\infty} \left((t-u)^2 + r^2\right)^{-\frac{s}{2}}\text{exp} \left(-\frac{r^2}{2(V(t)-V(u))}\right)\, dr \\[1em]
&= \frac{1}{2\sqrt{2\pi}} \int_{0}^{\infty} \left((t-u)^2 + (V(t)-V(u))w\right)^{-\frac{s}{2}}\,w^{-\frac{1}{2}}\,e^{-\frac{w}{2}}\, dw \\[1em]
&\lesssim \int_{0}^{\infty} \left((t-u)^2 + (V(t)-V(u))w\right)^{-\frac{s}{2}}\,w^{-\frac{1}{2}}\, dw.
\end{align*}
Now, since $V(t)$ is assumed to be strictly increasing we may split the integral
\begin{align*}
& \int_{0}^{\infty} \left((t-u)^2 + (V(t)-V(u))w\right)^{-\frac{s}{2}}\,w^{-\frac{1}{2}}\, dw \\[1em]
&\le\,\Bigg(\int_{0}^{\frac{(t-u)^2}{V(t)-V(u)}} (t-u)^{-s}\,w^{-\frac{1}{2}}\, dw + \int_{\frac{(t-u)^2}{V(t)-V(u)}}^{\infty} \left((V(t)-V(u))w\right)^{-\frac{s}{2}}\,w^{-\frac{1}{2}}\, dw \Bigg) \\[1em]
&\lesssim \frac{|t-u|^{-s+1}}{\sqrt{|V(t)-V(u)|}},
\end{align*}
and this completes the proof.
\end{proof}

\medskip

\noindent{\it Proof of Theorem \ref{hausmainresult}}. Now that the upper bound has been proved, it remains to show that the lower bound holds. Also, by considering a projection, the Hausdorff dimension of the graph is at least 1, and so we just need to consider the case where $2-\frac{\alpha_{\ast}}{2}\ge 1$ (i.e. $\alpha_{\ast}\le 2$).  Let $\epsilon>0$. Then we can find $t\in J$ such that
$$
 \alpha_{\ast}\le \alpha_{\ast}(t) \le \alpha_{\ast}+\epsilon.
$$
Since $\alpha_{\ast}\ge 1$, by definition of $\alpha_*(t)$ see that for all $\eta>0$,  there exists $\delta_t > 0$ such that
\begin{equation}\label{eq1.1}
|u_1-u_2|^{(1+\eta)\alpha_*(t)} \leq |V(u_1)-V(u_2)|, \quad\forall u_1, u_2\in I(t, \delta_t).
\end{equation}
Note that from monotonicity of dimension, 
$$
\mbox{dim}_H ({\mathcal G}(X,J)) \ge \mbox{dim}_H ({\mathcal G}(X,I(t,\delta_t))).
$$
 Let $J_1 = J\cap I(t,\delta_t)$. Define a measure $\mu$ on the graph via 
$$\mu_\mathcal{G}(E) = m \{t\in J_1\,:\,(t, X_t)\in E\}$$
where $m$ is Lebesgue measure. Then, using Lemma \ref{boundonexpec} and (\ref{eq1.1}), for all $s>0$ we have
$$
\begin{aligned}
\mathbb{E} I_s (\mu_\mathcal{G}) =& \iint \mathbb{E}\left[\frac{1}{\left(|v-u|^2 + |X_v - X_u|^2\right)^{\frac{s}{2}}} \right]\, dv \;du\\
 \leq& C \iint_{J_1\times J_1} \frac{|v-u|^{-s+1}}{\sqrt{|V(v)-V(u)|}}\, dv \;du
 \\
\leq& C \iint_{J_1\times J_1} |v-u|^{-s+1-\frac{(1+\eta)\alpha_*(t)}{2}}\, dv \;du\\
\end{aligned}
$$
which is finite if $s-1 + \frac{(1+\eta)\alpha_*(t)}{2}<1$. Since this holds for all $\eta>0$, we have 
$$\dim_H\mathcal{G}(X,J))\geq \dim_H\mathcal{G}(X,J_1)\geq 2 - \frac{\alpha_*(t_1)}{2}\geq 2 - \frac{\alpha_*+\epsilon}{2}.$$
Letting $\epsilon\to 0$, the proof is complete. 
\qquad$\Box$

\medskip

\subsection{Fractional Brownian Motion.} Our results in this section works as well to fractional Brownian motion. Let $0<H<1$, the fractional Brownian motion with Hurst index $H$, denoted by $(B^H_t)$, is a continuous stochastic processes such that 
$$
B^H_t-B^H_s \sim N(0, |t-s|^{2H}).
$$
We notice that the above proof works exactly the same for $X^H_t  = B^H_{V(t)}$ and the following theorem can easily be obtained. 

\begin{theorem}
 Let $(X^H_t)_{t\in J}$ be the continuous process such that $X^H_t = B^H_{V(t)}$ where $V$ is a strictly increasing function on $J$. Then almost surely
Then, almost surely
\begin{equation}
\max\left\{2 -  H\alpha_{\ast}, 1\right\} \leq \dim_H \ \mathcal{G}(X, J)\leq 2 -H\alpha^{\ast}.
\end{equation}
where ${\alpha}_{\ast} = {\alpha}_*(J)$ and ${\alpha}^{\ast} = {\alpha}^{\ast}(J)$.
\end{theorem}

\medskip

The upper bound follows from the fact that fractional Brownian motion is H\"{o}lder to all exponents less than $H$, while the lower bound is to observe that (\ref{eq_bound_expect}) is replaced by 
\begin{equation}\label{eq_bound_expect_frac}
\mathbb{E}\left[\frac{1}{\left(|t-u|^2 + |X_t - X_u|^2\right)^{\frac{s}{2}}} \right] \lesssim \frac{|t-u|^{-s+1}}{{|V(t)-V(u)|^H}}.
\end{equation}
We will omit the detail of the proof.

\medskip

\section{Prelude to Fourier dimension}

In the rest of the paper, we will provide an estimate of the almost sure Fourier dimension for graphs of continuous additive processes and also the graph of the Brownian staircase. In this section, we will give out the necessary results needed.

\subsection{Some results of Kahane.} The following general lemma was first used by Kahane \cite{K1985}, and is, by far, one of the most commonly used techniques for computing the almost sure Fourier dimension of random sets. We take it out as a lemma which may be useful for future study, Another version of this lemma can also be found in \cite[Lemma 6]{E2}. 
\medskip

\begin{lemma}\label{randmeaslem}
Let $\mu = \mu_\omega$ denote a random measure on $\R^d$ that is compactly supported almost surely and let $E$ be a Borel set of $\R^d$. Further, let $C, c, \gamma >0$ be such that
\begin{equation}\label{exptdecayassump}
\mathbb{E}\left[|\hat{\mu}(\xi)|^{2q}\right] \lesssim C^q q^{cq} |\xi|^{-\gamma q}, \; \forall q\in\Z^+, \; |\xi|>1.
\end{equation}
Then, almost surely
\begin{equation}\label{ftbound}
|\hat{\mu}(\xi)|^2 \lesssim (\log |\xi|)^c |\xi|^{-\gamma}, \; \text{as} \; |\xi|\rightarrow \infty.
\end{equation}
\end{lemma}
\begin{proof}
Let $\varepsilon>0$ and put $q_n = \left\lfloor \log |\e n|\right\rfloor$, $n\in\Z^d$. By Fubini theorem and (\ref{exptdecayassump}) we then have 
$$\mathbb{E}\sum_{n\in \Z^d\setminus\{0\}, \,|\e n| > 1} |\e n|^{-(d+1)} \left(\frac{|\hat{\mu}(\varepsilon n)|^2}{Cq_n^c |\e n|^{-\gamma}} \right)^{q_n} \leq \sum_{n\in \Z^d\setminus\{0\}, \,|\e n|>1} |\e n|^{-(d+1)} < \infty.$$
Thus, almost surely $$\lim_{|n|\rightarrow\infty} |\e n|^{-(d+1)} \left(\frac{|\hat{\mu}(\e n)|^2}{Cq_n^c |\e n|^{-\gamma}} \right)^{q_n} = 0$$ 
which means there exists $N\in\Z^+$ such that $|n|\geq N$ gives 
\begin{equation}\label{epdecay}
|\e n|^{-(d+1)} \left(\frac{|\hat{\mu}(\e n)|^2}{C q_n^c |\e n|^{-\gamma}} \right)^{q_n} < 1 \Longrightarrow |\hat{\mu}(\e n)|^2 \le C_{\omega}(\log (|\e n|))^c |\e n|^{-\gamma}, \forall n\in\N 
\end{equation}
for some random constant $C_{\omega}$. Now, let $\e_k\rightarrow 0$ and let $\Omega_k\subset \Omega$ be the event that (\ref{epdecay}) is satisfied with $\e = \e_k$. Let $\Omega' = \bigcap_{k=1}^{\infty} \Omega_k$, and see that because $\mathbb{P}(\Omega_k) = 1$, $\mathbb{P}(\Omega') = 1$. Since $\text{spt}\,\mu$ is compact, choose $\e_j$ so that $\text{spt}\,\mu \in \left(-\frac{\e^{-1}_j}{2}, \frac{\e^{-1}_j}{2}\right)^d$. Thus, because (\ref{epdecay}) holds with $\e = \e_j$, applying Lemma 1 in \cite[p.252]{K1985} completes the proof.
\end{proof}

\medskip

We also need the following theorem about image measure, also due to Kahane. 

\begin{theorem}\label{theorem_Kahane}
Let $\theta$ be a measure on $\R$ such that $\theta(I)  \lesssim |I|^{\gamma}$ for some $\gamma>0$. Let $\nu$ be the image measure of $\theta$ under Brownian motion i.e. $\nu(E) = \theta (B^{-1}(E))$. Then 
$$
{\mathbb E} \left[|\widehat{\nu}(\xi)|^{2q}\right] \le C^q q^q \cdot |\xi|^{-2\gamma q}.
$$ 
\end{theorem}
\medskip

\subsection{Stochastic calculus.} We will need some stochastic calculus in our estimation. We refer readers to \cite{KSbook,LeGall} for more details.  Let  $(Y_t,{\mathcal F}_t)$ be a continuous semi-martingale i.e.
$$
Y_t = M_t + T_t,
$$
where $M_t$ is a continuous local martingale and $T_t$ is a continuous adapted process of bounded variation. We have the following version of It\^{o}'s formula \cite[p.149]{KSbook}.

\begin{theorem}
Let $f:\R\to\R$ be a $C^2$ function. Then almost surely we have 
$$
f(Y_t)-f(X_0) = \int_0^t f'(Y_s)dM_s +\int_0^t f'(Y_s)dT_s + \frac12 \int_0^t f''(Y_s) d\langle M\rangle_s, \ \forall t>0
$$
where $\langle M\rangle_s$ is the quadratic variation process of $M_t$. If $M_t = B_t$, then $d\langle M\rangle_s = ds$.  
\end{theorem}

The term $\int_0^t f'(Y_s)dM_s$ is the stochastic integral and it. is also a continuous local martingale. For details about stochastic integrals, see \cite[Chapter 3]{KSbook}. The most important tool we need here is that the stochastic integral, as martingales, enjoys an important moment estimate, called the Burkholder-Davis-Gundy (BDG) inequality.

\begin{theorem}[Burkholder--Davis--Gundy]\label{bdg}
Let $(M_t)_{t\in\R^+}$ be a continuous local  martingale with $M_0 = 0$ almost surely and define $$M^*_t = \sup_{0\leq u \leq t} \,|M_u|.$$
Then, for any $1\leq q<\infty$
\begin{align}
\E[(M^*_t)^{2q}] \leq C_q \E[(\langle M \rangle_t)^q], \quad \forall t\in\R^+,
\end{align}
where $C_q \lesssim C^qq^{q}$ is a constant that depends only on $q\in[1,\infty)$ and $C$ is some positive constant.
\end{theorem}

The constant of the BDG inequality, according to \cite[equation (2.5) and (2.23)]{P}, was 
$$
C_q \le (\sqrt{2(4q+1)})^{2q} \lesssim (10)^qq^q. 
$$
(See footnote {\footnote{In \cite{FS2018}, they forgot to raise the power $q$ obtained from (2.5) in \cite{P}. However, as we will see, this will not affect their main conclusion.}}). We can obtain a weaker bound using some standard approach. The BDG inequality in the case $q\ge 1$ above follows from It\^{o}'s formula applied to $x^{2q}$ and Doob's submartingale's inequality. For details, see \cite[p.124-125]{LeGall}. In \cite[p.125]{LeGall} with $p = 2q$, the constant was found to be 
 $$
D_p =  \left( \left(\frac{p}{p-1}\right)^p\frac{p(p-1)}{2}\right)^{p/2}
 $$
 and $C_q = D_{2q}$. Notice that $\lim\limits_{p\to\infty}\left(\frac{p}{p-1}\right)^p = e$. Therefore, for some constant $C>0$,  $D_p\le C^p (p(p-1)/2)^{p/2}$. Hence, $C_q\le (2C)^qq^{2q}$.

\section{Fourier dimension}

Let $X_t = B_{V(t)}$ where $t\in J$ be a centered continuous additive process and we will assume that $V$ is strictly increasing so that $T = V^{-1}$ and is continuous. Fix a subinterval $J$ in $[0,1]$, the push-forward of Lebesgue measure by the graph is 
$$
\mu_{\mathcal G} (E)= m \{t\in J: (t,X_t)\in E\}.
$$
Then its Fourier transform is given by 
\begin{equation}\label{fouriertransrep}
\widehat{\mu_{\mathcal G} }(\xi) = \int_J e^{-2\pi i (\xi_1 t+\xi_2 X_t)} dt. 
\end{equation}
Using the function $T(s)$, we can rewrite this as
$$
\widehat{\mu_{\mathcal G} }(\xi_1,\xi_2) = \int_{V(J)} e^{-2\pi i (\xi_1 T(s)+\xi_2 B_s)}dT(s). 
$$
Our main result in this section is 

\begin{theorem}\label{thm_Fdim}
Let $X_t = B_{V(t)}$ where $t\in J$ be the centered continuous additive process and $V$ is strictly increasing with $T = V^{-1}$.  Let $J$ be a subinterval of $[0,1].$ Suppose that 
\begin{equation}\label{eqHolder}
|T(x)-T(y)| \lesssim |x-y|^{\gamma}, \forall x,y\in J.
\end{equation}
 Then 
$$
\mbox{\rm dim}_F {\mathcal G}(X,J) \ge \frac{2\gamma}{2+\gamma}.
$$
\end{theorem}

An increasing function $T$ induces a Lebesgue-Stieltjes measure, denoted by $dT$, such that $dT ([x,y]) = T(y)-T(x)$. Therefore (\ref{eqHolder}) implies that the measure $dT(I)\lesssim |I|^{\gamma}$.

\medskip

Our goal is to show that (\ref{exptdecayassump}) in Lemma \ref{randmeaslem} holds for the random measure $\mu_{\mathcal G} $. Then appealing to Lemma \ref{randmeaslem}, we will obtain our desired Fourier decay. Our approach is adapted from Fraser and Sahlsten \cite{FS2018}. We first define
\medskip

\begin{definition}[Horizontal and Vertical angles]\label{horzandvert}
Let $u>0, \varrho \in [1/2,1)$ and let $\theta_u = \min\{u^{-\varrho},\, \frac{\pi}{4}\}$. Then, we say that $\theta\in[0,2\pi)$ is a \textit{horizontal angle} if 
\begin{align}
\theta \in H_u := [0, \theta_u] \cup [\pi - \theta_u, \pi+\theta_u] \cup [2\pi-\theta_u, 2\pi)
\end{align}
and is a \textit{vertical angle} if
\begin{align}
\theta \in V_u := [0, 2\pi) \setminus H_u.
\end{align}
We now decompose $ \{\xi\in\R^2: |\xi|>1\} = {\mathcal V}_{\varrho}\cup{\mathcal H}_{\varrho}$ such that if  we write $\xi = u(\cos\theta_{\xi},\sin\theta_{\xi})$ for $\theta\in[0,2\pi)$ and $u = |\xi|$
$$
{\mathcal V}_{\varrho} = \{ \xi: \theta_{\xi}\in V_u\}, \ {\mathcal H}_{\varrho} = \{ \xi: \theta_{\xi}\in H_u\}.
$$
\end{definition}

In \cite{FS2018}, $\varrho$ was chosen to be $1/2$. In our proof, we will choose some larger $\varrho$ to obtain the necessary decay. The following is an adapted version of what appears as Lemma 3.2 in \cite{FS2018}. 

\begin{lemma}\label{trig}
Let $u>0$ and let $\theta \in H_u$ and $\phi\in V_u$. Then,
\begin{enumerate}
\item[$($I$\,)$] $|\sin\theta| \leq u^{-\varrho}$ \quad \text{and} \quad $|\cos\theta|\geq \frac{\sqrt{2}}{2}$,
\item[$($II$\,)$] $|\sin\phi| \geq \min\{\frac{2}{\pi}u^{-\varrho}, \, \frac{\sqrt{2}}{2}\}$.
\end{enumerate}
\end{lemma}
\begin{proof}
(I): It suffices to show the inequalities for $\theta\in [0, \theta_u]$. Since $\theta_u = \min\{u^{-\varrho},\, \frac{\pi}{4}\}$, for any $\theta\in [0,\theta_u]$ we have $$|\sin\theta| \leq \theta \leq \theta_u \leq u^{-\varrho}
 \ \mbox{and} \ |\cos\theta |\geq \cos\theta_u \geq \cos\frac{\pi}{4} = \frac{\sqrt{2}}{2}.$$

(II): It suffices to show the inequality for $\phi\in (\theta_u, \frac{\pi}{2}]$. To this end, observe that $\sin\phi \geq \frac{2}{\pi} \phi$ for $\phi \in [0, \frac{\pi}{2}]$ and as a result we immediately find (II).
\end{proof}

\medskip

\subsection{Vertical angle.} Our main result in this section is  

\medskip

\begin{proposition}\label{prop_vertical} Let $\mu$ be a finite Borel measure on $[0,1]$ and suppose that $T$ satisfies (\ref{eqHolder}).  Then 
$$
{\mathbb E} [|\widehat{\mu_{\mathcal G}}(\xi)|^{2q}]  \  \lesssim \  C^q q^q |\xi|^{-(2-2\varrho)\gamma q} 
$$
for all $\xi\in {\mathcal V}_{\varrho}$.
\end{proposition}

\medskip

This proposition follows from the following lemma. 
\begin{lemma}\label{lemG}
Let $\xi = (\xi_1, \xi_2)\in \R^2$. Then, for any $q\in \N$
\begin{equation}
\mathbb{E}\left[\left|\hat{\mu}_\mathcal{G}(\xi)\right|^{2q}\right] \leq \mathbb{E}\left[\left|\hat{\nu}(\xi_2) \right|^{2q} \right].
\end{equation}
where $\nu = \mu \circ X^{-1}$
\end{lemma}

\begin{proof}
By expanding the $2q$-moment into iterated integral, we see that
\begin{align*}
&\mathbb{E}\left[\left|\hat{\mu}_\mathcal{G}(\xi_1, \xi_2)\right|^{2q}\right] \\
&= \idotsint\limits_{[0,1]^{2q}} e^{-2\pi i \xi_1 \left(\sum_{k=1}^{q}(s_k - s'_k) \right)}\,\mathbb{E}\left[e^{-2\pi i \xi_2 \left(\sum_{k=1}^{q} (X_{s_k} - X_{s'_k})\right)}\right] \, ds_1\ldots ds_q ds'_1 \ldots ds'_q.
\end{align*}
Further, because $-2\pi \left(\sum_{k=1}^{q} (X_{s_k} - X_{s'_k})\right)$ is a Gaussian random variable, the expectation in the integrand is positive. We may take the absolute value of both sides of the above to find 
\begin{align*}
\mathbb{E}\left[\left|\hat{\mu}_\mathcal{G}(\xi_1, \xi_2)\right|^{2q}\right] \leq &  \idotsint\limits_{[0,1]^{2q}} \mathbb{E}\left[e^{-2\pi i \xi_2 \left(\sum_{k=1}^{q} (X_{s_k} - X_{s'_k})\right)}\right] \, ds_1\ldots ds_p ds'_1 \ldots ds'_p \\
=& \mathbb{E}\left[\left|\hat{\nu}(\xi_2) \right|^{2q} \right],
\end{align*}
completing the proof.
\end{proof}

\medskip

\noindent{\it Proof of Proposition \ref{prop_vertical}.} We first note that 
$$
\widehat{\nu}(\xi_2) =\int_{c}^{d} e^{-2\pi i \xi_2 X_t}dt  = \int_{0}^{V(1)} e^{-2\pi i \xi_2 B_s}dT(s).
$$
  Using Lemma {\ref{lemG}} and Theorem \ref{theorem_Kahane} with $\theta = dT(s)$, we have that
$$
{\mathbb E} [|\widehat{\mu_{\mathcal G}}(\xi)|^{2q}] \le  C^q q^q |\xi_2|^{-2\gamma q}.
$$
If $\xi\in {\mathcal V}_{\varrho}$, from Lemma \ref{trig} (ii), we have two cases (i) $|\sin\theta_{\xi}| \ge \frac{2}{\pi}u^{-\varrho}$ and (ii) $|\sin \theta_{\xi}| \ge \sqrt{2}/2$. In the first case, 
$$
\begin{aligned}
{\mathbb E} [|\widehat{\mu_{\mathcal G}}(\xi)|^{2q}] \le & C^q q^q \cdot |\xi_2|^{-2\gamma q}\\ 
\lesssim&  C^q q^q |\xi|^{(2\varrho-2)\gamma q} .\\
\end{aligned}
$$
In the second case, we have 
$$
{\mathbb E} [|\widehat{\mu_{\mathcal G}}(\xi)|^{2q}] \lesssim  C^q q^q  |\xi|^{-2\gamma q}. 
$$
As $\varrho\in(0,1)$, $|\xi|^{2\varrho-2}>|\xi|^{-2}$ for all $|\xi|>1$, so our proposition follows.  \qquad$\Box$

\medskip
   
\subsection{Horizontal angle.} We now estimate the horizontal angle cases $\xi\in {\mathcal H}_{\varrho}$.  Let us define
$$
Z_t = \xi_1 t+\xi_2 X_t
$$
and 
$$
Y_t =-2\pi ( \xi_1 t+ \xi_2 X_t )= -2\pi  Z_t
$$
where $\xi = u (\cos\theta,\sin\theta)$ for $u= |\xi|$. Define now the random time $\tau$. 

\begin{definition}\label{randomt}
Let $\xi = u(\cos\theta, \sin\theta)\in \R^2\setminus\{0\}$ with $\theta\in H_u$ and let $J = [c,d]$ be a subinterval of $[0,1]$. Then, we define the random time $\tau=\tau_\omega(\xi)\in J$ as
\begin{equation}
\tau = \min \{t\in J\,:\, Y_t = \kappa^*\},
\end{equation}
where
\[
\kappa^* =  
\begin{cases} 
	-2\pi \left\lceil   Z_d-Z_c \right\rceil +Y_c, & \text{if}\; Y_d \ge Y_c, \\ 
	-2\pi \left\lfloor  Z_d-Z_c   \right\rfloor +Y_c, & \text{if}\; Y_d <  Y_c. 
\end{cases}
\]
\end{definition}
In the above, $\lceil \cdot \rceil$ and $\lfloor\cdot\rfloor$ denotes the upper and lower floor functions. 

\begin{lemma}
The random time $\tau$ exists almost surely. 
\end{lemma}
\begin{proof}
Suppose that $Y_d \ge Y_c$.  Then $Z_d-Z_c \le 0$.  This implies that $0\ge \left\lceil  Z_d-Z_c \right\rceil \ge Z_d-Z_c$. Hence, 
$$
0\le -2\pi \left\lceil  Z_d-Z_c \right\rceil \le Y_d-Y_c.
$$
Thus, $\kappa^{\ast}\in [Y_c,Y_d]$. By the intermediate value theorem, as $X_t$ has continuous sample paths almost surely, $\tau$ exists almost surely. In a similar consideration, the second case also holds.  
\end{proof}
%

We now make use of $\tau$ via decomposing $\hat{\mu}_\mathcal{G}(\xi)$ using (\ref{fouriertransrep}) in order to write
\begin{equation}\label{eqmu_dec}
\hat{\mu}_\mathcal{G}(\xi) = \underbrace{\int_{c}^{\tau} e^{i Y_t}\,dt}_{\hat{\mu}_\mathcal{G}{\big|}_{[c, \tau]}(\xi)} + \underbrace{\int_{\tau}^{d} e^{i Y_t}\,dt}_{\hat{\mu}_\mathcal{G}{\big|}_{\tau,d]}(\xi)}, \quad \xi 
\in{\mathcal H}_{\varrho}
\end{equation}
which then gives $|\hat{\mu}_\mathcal{G}(\xi)| \leq \left| \int_{c}^{\tau} e^{i Y_t}\,dt\right| + \left|\int_{\tau}^{d} e^{i Y_t}\,dt\right|$. This reduces bounding $|\hat{\mu}_\mathcal{G}(\xi)|$ to bounding $|\hat{\mu}_\mathcal{G}{\big|}_{[c,\tau]}(\xi)|$ and $|\hat{\mu}_\mathcal{G}{\big|}_{[\tau,d]}(\xi)|$ individually.  

\medskip

\begin{proposition}\label{PropH} With respect to the above notations, there exists a constant $C> 0$ such that for all $\xi\in{\mathcal H}_{\varrho}$, 
\begin{equation}
\E\left[|\hat{\mu}_\mathcal{G}(\xi)|^{2q}\right] \lesssim C^q q^q |\xi|^{-(4\varrho -2)q}
\end{equation}
for some $C>0$ depending only on $V$. 
\end{proposition}

\begin{proof}
As we have observed in \eqref{eqmu_dec}, we will establish the following two estimates for $\xi\in{\mathcal H}_{\varrho}$.
\begin{enumerate}
\item $\E\left[\left|\int_{\tau}^{d} e^{i Y_t}\,dt\right|^{2q}\right]  \lesssim C^q  q^q \cdot |\xi|^{-2\varrho q},$
\item $\E\left[\left|\int_c^{\tau}e^{i Y_t}\,dt\right|^{2q}\right]  \lesssim   C^q q^{q} |\xi|^{-(4\varrho-2)q}$
\end{enumerate}
for some constant $C>0$ depending only on $V$. Note that for $\varrho\in [1/2,1)$, we always have $2\varrho \ge 4\varrho-2$. The second term always possess a slower decay. Therefore,  our proposition follows. Let us establish the estimate now. 

\medskip

\noindent (1). Recall that  $Z_t = u(t\cos\theta + X_t\sin\theta)$ and that $Y_t = -2\pi Z_t$. By the definition of $\tau$  in (\ref{randomt}) we have $Z_{\tau} = \lceil Z_d-Z_c\rceil + Z_c$ if $Y_d\ge Y_c$ and $Z_{\tau} = \lfloor Z_d-Z_c\rfloor + Z_c$ if $Y_d<Y_c$. In both cases, the following inequality always holds
$$
Z_d-1 \le Z_{\tau} \le Z_d+1.
$$
Expanding it out, we have 
$$
u (d \cos \theta+ X_d\sin\theta)-1 \le u (\tau \cos\theta+ X_{\tau}\sin \theta) \le u (d \cos \theta+ X_d\sin\theta)+1. 
$$
If $\cos\theta>0$, we can use the first inequality to obtain
\begin{equation}\label{eq5.7}
\tau \geq d+ X_d \frac{\sin\theta}{\cos\theta} - X_{\tau}\frac{\sin\theta}{\cos\theta} - \frac{1}{u\cos\theta},
\end{equation}
and if $\cos\theta<0$ we can use the second inequality to obtain
\begin{equation}\label{eq5.8}
\tau \geq d+ X_d \frac{\sin\theta}{\cos\theta} - X_{\tau}\frac{\sin\theta}{\cos\theta} + \frac{1}{u\cos\theta}.
\end{equation}
Now, define the random variable $M = \max_{t\in [0,1]} |X_t|$. Since $X$ is almost surely continuous, $M<\infty$ almost surely. This combined with {Lemma \ref{trig}}, \eqref{eq5.7} and \eqref{eq5.8} above yields 
$$\tau\geq d - 2\sqrt{2} M u^{-\varrho} - \frac{\sqrt{2}}{u},$$
and consequently we find
$$
\left|\int_{\tau}^{d} e^{i Y_t}\,dt\right| \leq  d - \tau = 2\sqrt{2} Mu^{-\varrho} + \frac{\sqrt{2}}{u} \leq  \sqrt{2}(2M+1) |\xi|^{-\varrho}.
$$
Taking expectation and applying the elementary inequality $|a+b|^{2q}\le 2^{q}(|a|^q+|b|^q)$ gives
$$
\E\left[\left|\int_{\tau}^{d} e^{i Y_t}\,dt\right|^{2q}\right] \le |\xi|^{-2\varrho q} \E [|2\sqrt{2}M+2|^{2q}] \lesssim C^q \cdot \E[|M|^{2q}] \cdot |\xi|^{-2\varrho q}.
$$
Notice that 
$$
M = \max_{t\in [0,1]} |X_t|= \max_{s\in [0,V(1)]}|B_s|.
$$
The probability distribution of $M$ is known and we have  ${\mathbb P}(M\ge b) \le  \sqrt{\frac{V(1)}{2\pi}} \frac4{b} e^{-b^2/2V(1)}$ \cite[p.96]{KSbook}. This implies that 
$$
\E[|M|^{2q}] \lesssim \int_0^{\infty} b^{2q-1} e^{-b^2/2V(1)}db \lesssim  (V(1))^q\int_0^{\infty} b^{2q} e^{-b^2/2}db.
$$
By a standard moment estimate of the normal distribution,  $\int_0^{\infty} b^{2q} e^{-b^2/2}db \lesssim  (2q-1)!! = (2q-1)(2q-3)....3.1 \lesssim (2q)^q$. Hence, $\E[|M|^{2q}] \lesssim C^qq^q$ for some constant $C>0$ depending only on $V$.
\medskip

\noindent (2). We now proceed to study the other integral from $0$ to the random time $\tau$. Notice that It\^{o}'s formula extends easily to complex valued $C^2$ functions by simply considering real and imaginary parts. We now take  $Y_s =bT_s+ \sigma B_s $ where $\sigma B_s$ is the martingale and $bT_s$ is the finite variation process and also consider  
$$
f(x) = e^{i x}
$$
for some constant $b,\sigma\in \R$. We will put 
$$
b = -2\pi u \cos\theta, \ \sigma = -2\pi u \sin\theta.
$$
 Then we have for $R = V(\tau)$ and $S = V(c)$
$$
\begin{aligned}
& e^{i (bT(R)+\sigma B_R)} -e^{i (bT(S)+\sigma B_S)} \\
 =&  ib \int_S^R e^{i (bT(s)+\sigma B_s)}dT(s)+ i\sigma \int_S^Re^{i (bT(s)+\sigma B_s)} dB(s)-\frac12\sigma^2 \int_S^R e^{i (bT(s)+\sigma B_s)} ds
\end{aligned}
$$
Note that the left hand side above is equal to  $ e^{i(b\tau+\sigma X_{\tau})} - e^{i (bc+ \sigma X_c)}  = e^{iY_{\tau}} - e^{iY_c}= 0$ by the definition of $\tau$. Hence, our formula reduces to 
$$
- ib \int_c^{\tau} e^{i (bt+\sigma X_t)}dt =  i\sigma \int_S^Re^{i (bT(s)+\sigma B_s)} dB(s)-\frac12\sigma^2 \int_{c}^{\tau} e^{i (bt+\sigma X_t)} dV(t)
$$
(See footnote {\footnote{For standard Brownian Motion, $dV = dt$ and we obtain a further factorization, allowing Fraser and Sahlsten \cite{FS2018} obtaining a sharp result in the case of Brownian Motion}}).
Hence, applying the elementary inequality $|a+b|^{2q}\le 2^{2q} (|a|^q+|b^q|)$
\begin{equation}\label{eq012}
\begin{aligned}
&  \ \E\left[\left|\int_c^{\tau} e^{i (bt+\sigma X_t)}dt\right|^{2q} \right]    \\
 \le &  \  2^{2q} \left(\frac{\sigma^{2q}}{b^{2q}} \E \left[ \left|\int_S^Re^{i (bT(s)+\sigma B_s)} dB(s)\right|^{2q}\right] + \frac{\sigma^{4q}}{2^{2q}b^{2q}} \E\left[\left|\int_c^{\tau} e^{i (bt+\sigma X_t)} dV(t)\right|^{2q}\right] \right). \\
\end{aligned}
\end{equation}
 Let $f(s) =\left\{\begin{array}{ll} \cos (bT(s)+\sigma B_s) & \mbox{if}  \ s\in[V(c),V(d)] \\ 0 & \mbox{if} \  s\in[0,V(c)] \end{array}\right.$.  With this extension, the It\^{o} integral 
$$
\int_{V(c)}^{t} \cos (bT(s)+\sigma B_s)  dB(s)= \int_0^{t} f(s)dB_s
$$
for all $t\in[0,V(1)]$. Moreover, this process is a martingale. By the Burkholder-Davis-Gundy inequality, for some constant $C_q\lesssim q^q$ (Theorem \ref{bdg}),
$$
\begin{aligned}
\E \left[ \left|\int_{V(c)}^{V(\tau)} \cos (bT(s)+\sigma B_s)  dB(s)\right|^{2q}\right]  \le&  \  \E \left[  \sup_{t\in[0,V(1)]}\left|\int_0^t f(s) dB(s)\right|^{2q}\right]   \\
\le& \ C_q \E \left[ \left|\int_{V(c)}^{V(1)}\cos^2 (bT(s)+\sigma B_s)  ds\right|^{q}\right]  \\
\lesssim& \  C^qq^{2q}(V(1))^q   \ \ \    \ (\mbox{because} \cos^2(...)\le 1).
\end{aligned}
$$
The same also holds with $\cos$ replaced by $\sin$. Hence, by using $e^{i\theta} = \cos\theta+i\sin\theta$ and  $|a+b|^{2q}\le 2^{2q} (|a|^q+|b^q|)$,
\begin{equation}\label{eq456}
\E \left[ \left|\int_{V(c)}^{V(\tau)}e^{i (bT(s)+\sigma B_s)} dB(s)\right|^{2q}\right] \lesssim 2^q q^{q} (V(1))^q.
\end{equation}
Also, by triangle inequality, we have 
\begin{equation}\label{eq123}
\E\left[\left|\int_c^{\tau} e^{i (bt+\sigma X_t)} dV(t)\right|^{2q}\right]  \le \E\left[\left(\int_c^{\tau}  dV(t)\right)^{2q}\right]  \le (V(\tau))^{2q}\le V(1)^{2q}.
\end{equation} 
Putting (\ref{eq456}) and (\ref{eq123}) into (\ref{eq012}), we obtain 
$$
\E\left[\left|\int_c^{\tau} e^{i (bt+\sigma X_t)}dt\right|^{2q} \right]  \lesssim \left(\frac{\sigma}{b}\right)^{2q}\cdot 2^{3q} q^{q} (V(1))^q +   \left(\frac{\sigma^2}{b}\right)^{2q} \cdot2^{q} V(1)^{2q}.
$$

Now, we apply Lemma \ref{trig} to obtain 
$$\frac{\sigma^{2q}}{b^{2q}}  = \frac{(-2\pi u\sin\theta)^{2q}}{(-2\pi u\cos\theta)^{2q}} \leq 2^q u^{-2\varrho q}\quad \text{and} \quad
\frac{\sigma^{4q}}{b^{2q}} = \frac{(-2\pi u\sin\theta)^{4q}}{(-2\pi u\cos\theta)^{2q}} \leq (2\pi^2)^qu^{-(4\varrho - 2)q},$$
and note that because $2\varrho \ge  4\varrho - 2\ge 0$ for $\varrho\in[1/2,1)$, we have 
$$
\E\left[\left|\int_0^{\tau} e^{i (bt+\sigma X_t)}dt\right|^{2q} \right]  \lesssim  C^q q^{q} |\xi|^{-(4\varrho-2)q}
$$
for some constant $C>0$.  Hence, our proposition follows. 
\end{proof}

\medskip

\noindent{\it Proof of Theorem \ref{thm_Fdim}.} Let $(X_t)_{t\in[0,1]}$ be a centered continuous additive process; we write $X_t = B_{V(t)}$ for some strictly increasing function $V$. Then $T = V^{-1}$ is a continuous and strictly increasing function satisfying 
$$
|T(x)-T(y)| \lesssim |x-y|^{\gamma}.
$$
As $T$ induces a Lebesgue-Stieltjes measure $\theta (I) = T(y)-T(x)$, $I = [x,y]$, we know that $\theta(I)\lesssim |I|^{\gamma}$ holds. Hence, Proposition \ref{prop_vertical} can be applied.  We notice that by Proposition \ref{prop_vertical} and \ref{PropH}, there exists a constant $C>0$ such that for all $\xi\in\R^2$, 
$$
\E\left[|\hat{\mu}_\mathcal{G}(\xi)|^{2q}\right] \lesssim C^q q^{q} |\xi|^{-\beta q}
$$
where $\beta = \min \{4\varrho-2, (2-2\varrho)\gamma\}$. By Lemma \ref{randmeaslem}, dim$_F {\mathcal G}(X,J)\ge \beta$.  We now maximize $\beta$ over $\varrho\in[1/2,1)$. We find that the maximum occurs when $\varrho = \frac{1+\gamma}{2+\gamma}$ and $\beta  = \frac{2\gamma}{2+\gamma}$.
\qquad$\Box$
\medskip

\noindent{\it Proof of Theorem \ref{fouriermainresult}}. Let $t\in J$ be the point with positive upper local uniform H\"{o}lder index. Then for any $\eta>0$ sufficiently small, we can find a sufficiently small interval $I(t,\delta_t)$ such that 
$$
|T(x)-T(y)| \le |x-y|^{(1-\eta)\alpha^{\ast}(t)}
$$
Applying Theorem \ref{thm_Fdim}, we obtain immediately that 
$$
\mbox{dim}_F {\mathcal G}(X,I(t,\delta_t))\ge  \frac{2(1-\eta)\alpha^{\ast}(t)}{2+(1-\eta)\alpha^{\ast}(t)}.
$$
Taking $\eta\to 0$ and using that Fourier dimension is monotone yields, 
$$
\mbox{dim}_F {\mathcal G}(X,J) \ge \mbox{dim}_F {\mathcal G}(X,I(t,\delta_t))\ge  \frac{2\alpha^{\ast}(t)}{2+\alpha^{\ast}(t)}.
$$
This completes the proof. 
\qquad$\Box$

\medskip

\section{Examples and Proof of Corollary \ref{corollaryhaus}}

We now give a way on how to use our theorems to determine the Hausdorff dimension of continuous additive processes with a given variance function. We first notice that there are two scenarios according to the decomposition in (\ref{decomposition})
\begin{enumerate}
\item there are only finitely many open intervals on which $V$ is constant. 
\item there are countably many disjoint open intervals on which $V$ is constant. 
\end{enumerate}
 In the first case, we can without loss of generality assume that $V$ is strictly increasing. This is because $E_V$ will simply be a finite union of closed intervals in $J$ and inside each interval, $V$ is strictly increasing and we can determine the Hausdorff dimension on each small intervals. In the second case, as $V$ is a continuous function, $E_V$ cannot contain isolated points. Hence, $E_V$ is perfect. If we have an interval inside $E_V$, we can study it separately as in the first case. Therefore, we can assume that  $E_V$ is a totally-disconnected Cantor-type set. It is difficult to study this case in full generality. We will  study $E_V$ as the middle-third Cantor set in the next section.

\medskip

Focusing on the first case, we recall that if $V$ is locally uniformly bi-Lipschitz at a point $t$, i.e. there exists $\delta>0$, $c,C>0$ depending on $\delta$ such that 
$$
c |u-v| \le |V(u)-V(v)| \le C|u-v|
$$
for all $u,v\in I(t,\delta)$, then it is immediately seen that $\alpha^{\ast} (t)= \alpha_{\ast}(t) = 1$.

\medskip

\noindent{\it Proof of Corollary \ref{corollaryhaus}.} Let $V$ be a locally bi-Lipschitz function on $J \setminus\{t_1,...,t_N\}$, where $t_1,..,t_N$ are the exceptional points for which $V$ is not locally bi-Lipschitz. We first notice that such $V$ must be strictly increasing on $J$. Otherwise, if there exists $u_1\ne u_2$ such that $V(u_1) = V(u_2)$, then $V$ must be constant on $[u_1,u_2]$. However, there must be a point such that $V$ is locally uniformly bi-Lipschitz. The lower bound implies $V$ cannot be a constant in the neighborhood of that point, which is a contradiction. 

\medskip

Let $I_{i,k} = \left[t_i-\frac1n,t_i-\frac1{n+1}\right]\cup \left[t_i+\frac1n,t_i+\frac1{n+1}\right]$. Take $n_1,...n_N$ so that $I_{i,n_i}$ are all disjoint. Then 
$$
J\setminus \{t_1,...,t_N\} = \left(\bigcup_{i=1}^N\bigcup_{k= n_i}^{\infty} I_{i,k}\right) \cup J',
$$ 
 where $J'$ is the complement of the unions in $J\setminus\{t_1,...,t_N\}$. Note that 
 $$
 {\mathcal G}(X, J) =  \left( \bigcup_{i=1}^N\bigcup_{k= n_i}^{\infty}{\mathcal G}(X, I_{i,k}) \right)\cup {\mathcal G}(X,J')  \cup {\mathcal G}(X,\{t_1,...,t_N\}). 
  $$ 
  Since all points in $I_{i,k}$ and $J'$ are locally bi-Lipschitz, we find that $\alpha_{\ast}=1$. Hence, dim$_H{\mathcal G}(X, I_{i,k}) = 3/2$ and dim$_H$ ${\mathcal G}(X, J') = 3/2$ by Theorem \ref{hausmainresult}. While the dimension of graph of finitely many points must be zero, by the countable stability of Hausdorff dimension, dim$_H{\mathcal G}(X,J) = 3/2$ holds. Using Theorem \ref{thm_Fdim}, the Fourier dimension is at least 2/3.  
  \qquad$\Box$

 \medskip
 
 In the following two examples, the Hausdorff and Fourier dimension estimates can be determined using Corollary \ref{corollaryhaus}, and they additionally illustrate that the upper and lower uniform H\"{o}lder indices may be different at critical points.

\begin{example}\label{example6.1}  Let $V(t) = t^6$, the Hausdorff dimension of the graph was determined to be $3/2$. However, one can calculate that $\alpha_{\ast}(0) = 6$. Indeed, 
 if $\alpha<6$, then
 $$
  \sup_{u_1, u_2 \in I(0,\delta)} \frac{|u_1 - u_2|^\alpha}{|V(u_1) - V(u_2)|} \ge \frac{u_1^{\alpha}}{u_1^6} \ge \frac{1}{\delta^{6-\alpha}}\to\infty
 $$
 as $\delta\to 0$. This shows that $\alpha_{\ast}(0) \ge 6$. On the other hand, let $\alpha>6$ and $u_1<u_2<\delta$, we have
 $$
 \frac{|u_1 - u_2|^\alpha}{|V(u_1) - V(u_2)|}  \ = \ u_2^{\alpha-6}\cdot\frac{(1-(u_1/u_2))^{\alpha}}{1-(u_1/u_2)^6} \ \lesssim \ \delta^{\alpha-6}
 $$
 as $\lim_{x\to 1} \frac{(1-x)^{\alpha}}{1-x^6} = 0$, which means the expression$ \frac{(1-(u_1/u_2))^{\alpha}}{1-(u_1/u_2)^6}$ is bounded for all $u_1/u_2\in[0,1)$. This shows that $\alpha_{\ast}(0) = 6$. 
  This indicates, as in the graph shown, that the trajectories are infinitetestimally flat near $0$. The example also holds for all $V(t) = t^{\beta}$ where $\beta>1$, in which $\alpha_{\ast}(0) = \beta$.
  \end{example}
  
  \begin{figure}[h]
\begin{center}
	\includegraphics[width=10cm]{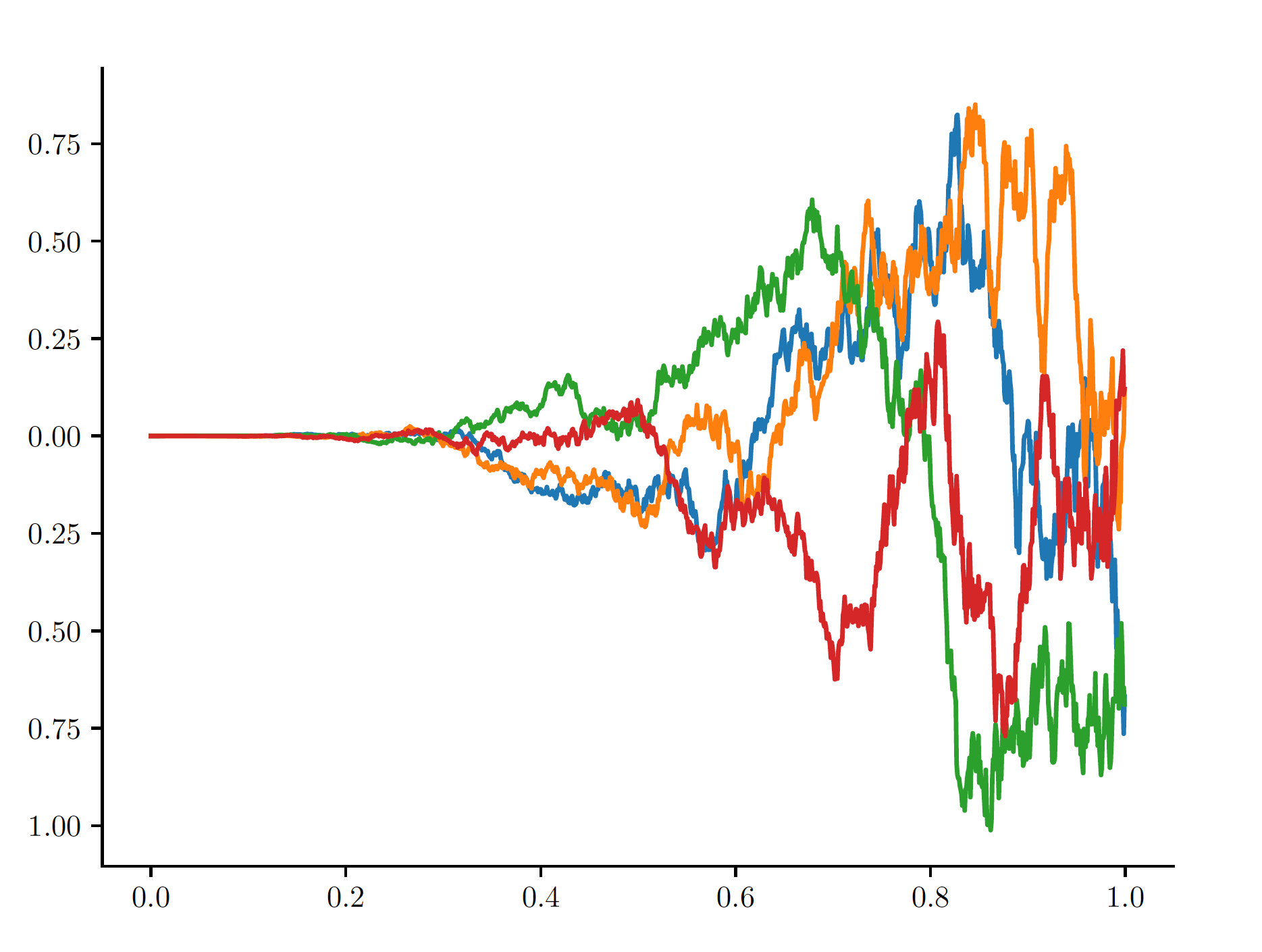}
\end{center}
\caption{Four sample paths of $(B_{t^6})_{t\in[0,1]}$}
\end{figure}
  
  \medskip
  
\begin{example}   If we now consider $V(t) = \sqrt{t}$.  We will show that $\alpha^{\ast} (0)= 1/2$ and $\alpha_{\ast}(0) = 1$. To see this, we note that for all $\alpha>1/2$  and $u_1,u_2\in[0,\delta)$,
$$
\sup_{u_1,u_2\in I(0,\delta)}\frac{|V(u_1)-V(u_2)|}{|u_1-u_2|^{\alpha}}\ge \frac{\sqrt{u_1}}{u_1^{\alpha}} = \frac{1}{u_1^{\alpha-1/2}} \ge \frac{1}{\delta^{\alpha-1/2}} \to\infty
$$
as $\delta\to 0$. Hence, $\alpha^{\ast}(0) \le 1/2$.  On the other hand, for all $\alpha<1/2$ and $u_1<u_2<\delta$,  
$$
\frac{|V(u_1)-V(u_2)|}{|u_1-u_2|^{\alpha}}  = u_2^{1/2-\alpha}\cdot \frac{1-(u_1/u_2)^{1/2}}{(1-u_1/u_2)^{\alpha}} \lesssim \delta^{1/2-\alpha}\to 0
$$
because $ \lim_{x\to 1}\frac{1-x^{1/2}}{(1-x)^{\alpha}} = 0$, which means that $ \frac{1-(u_1/u_2)^{1/2}}{(1-u_1/u_2)^{\alpha}}$ is bounded for all $u_1/u_2\in[0,1)$. Hence, $\alpha^{\ast} = 1/2$. 

\medskip
 
We now notice that, from Proposition \ref{propindex}, 
$$
\alpha_{\ast}(0) = \frac1{\beta^{\ast}(0)}
$$ 
 where $\beta^{\ast}(0)$ is the upper local H\"{o}lder index of $t^2$ at $t = 0$. Using mean value theorem,  we can deduce easily that $\beta^{\ast}(0) = 1$. Therefore, $\alpha_{\ast}(0) = 1$.
 \end{example}

   \begin{figure}[h]
\begin{center}
	\includegraphics[width=10cm]{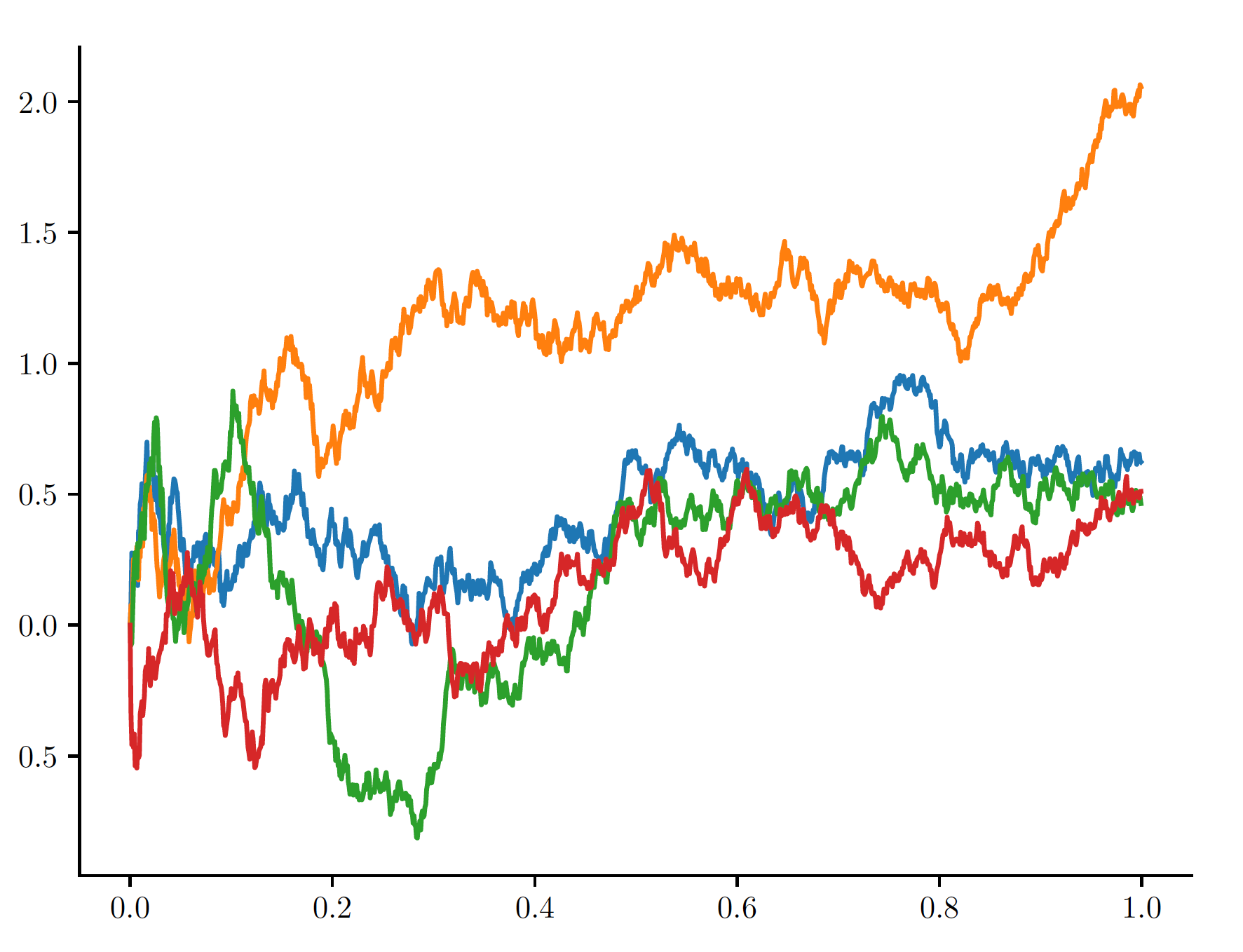}
\end{center}
\caption{Four sample paths of $(B_{\sqrt{t}})_{t\in[0,1]}$}
\end{figure}
%
%

\medskip

The following example studied distribution functions from the Bernoulli convolution associated with the ``golden ratio", which is one of the most famous singularly continuous measures. For recent progresses about researches of Bernoulli convolution, readers may refer to \cite{Varju}.

\begin{example}\label{example6.3}
Let $\rho = \left(\frac{1+\sqrt{5}}{2}\right)^{-1}$ be the inverse of the golden ratio. Consider the map
$$
S_{1}(x) = \rho x, S_2(x) = \rho x+(1-\rho).
$$
The Bernoulli convolution is the unique self-similar measure $\mu$ such that 
$$
\mu (E) = \frac12 \mu (S_1^{-1}(E))+\frac12 \mu (S_2^{-1}(E)), \forall E \ \mbox{Borel}.
$$
Note that $\mu$ is compactly supported on $[0,1]$ and additionally $\mu$ is well-known to be a singular measure without atoms. Therefore, $V(t) = \mu (-\infty,t]$ is a strictly increasing continuous function on $[0,1]$. In studying the multifractal analysis of $\mu$, it was known (see e.g. \cite[Proposition 2.2]{FengLau}) that there exists $\delta>0$ and $s_2>s_1>0$ such that for all $x\in[0,1]$ and $0<r<\delta$
$$
r^{s_1} \lesssim\mu(x-r,x+r)\lesssim r^{s_2}.
$$
In particular, we know that $s_2<1$ (otherwise $\mu$ is not singular) and $s_1 = \frac{\log 2}{-\log\rho} \sim 1.4404...$ (from the last line of \cite[Proposition 2.2]{FengLau}). Translating this result to $V$, we know
$$
|x-y|^{s_1}\lesssim|V(x)-V(y)|\lesssim |x-y|^{s_2}
$$
for all $x,y\in[0,1]$ and $|x-y|\le \delta$. Hence, $ \alpha^{\ast}\ge s_2$ and $\alpha_{\ast}\le s_1$, and as a result the Hausdorff dimension of the graph of $X_t = B_{V(t)}$ is between $2-s_1/2$ and $2-s_2/2$. Let $\beta = 1/s_1\sim0.6942...$; then, the Fourier dimension of the graph of $X_t$ is at least $\frac{2\beta}{2+\beta} \sim 0.5154....$
\end{example}

\section{$L^q$ spectrum and Self-similar measures}\label{section-self-similar}

In the next two sections, we will prove our multifractal result. We will first define the terminologies and provide some lemmas necessary for the main proof.

 \subsection{$L^q$ spectrum for continuous functions.} Given a closed ball $B$, oscillation of $f$ on $B$ is defined to be 
 $$
 \mbox{Osc}_f(B) = \sup_{s,t\in B} |f(s)-f(t)|
 $$
 The $L^q$-spectrum of a continuous function $f$ is defined to be 
 $$
 \tau_f(q) = \liminf_{r\to0} \frac{1}{\log r} \log \left( \sup_{{\mathcal B}} \sum_{B\in {\mathcal B}} (\mbox{Osc}_f (B))^q \right)
 $$
 where supremum is taken over all families of disjoint closed balls ${\mathcal B}$ inside the domain of $f$ where $ \mbox{Osc}_f (B)>0$
 for all $B\in{\mathcal B}$.

\medskip

 Let 
 $$
 X^H(t) = B^H_{V(t)}
 $$
 where $B^H$ is the fractional Brownian motion with Hurst index $0<H<1$  and $V$ is a continuous and strictly increasing function on $[0,1]$ with $V(0) = 0$. Associate with $V$, there is a measure $\mu$ induced by $V$ given by $\mu[0,x] = V(x)$. We can define the $L^q$ spectrum for $\mu$ by 
 $$
 \tau_{\mu}(q) =  \liminf_{r\to0} \frac{1}{\log r} \log \left( \sup_{{\mathcal B}} \sum_{B\in {\mathcal B}} (\mu(B))^q \right)
 $$
  where supremum is taken over all families of disjoint closed balls ${\mathcal B}$ where each $B\in{\mathcal B}$ has the center inside the support of $f$. This is clear that $\tau_{\mu}(q) = \tau_V(q)$ and $\tau_V(1)  = 0$. It is also known that $\tau_V$  is a concave and increasing function on ${\mathbb R}$. Therefore, $\tau(q)<0$ if $q<1$.

\medskip


\medskip

Our goal is to compute the almost sure Hausdorff dimension of the graph of $X^H$. We first show that a generic upper bound exist for any  increasing function

\begin{lemma}\label{lemma_tau_upper_bound}
Almost surely, the continuous additive process $X(t) = B^H_{V(t)}$ satisfies 
$$
\overline{\mbox{\rm dim}}_B ({\mathcal G}(X, [0,1])) \le 1- \tau_V\left(H\right).
$$
\end{lemma}

\begin{proof}
 As the fractional Brownian motion is H\"{o}lder to all exponents less than $H$, we see that for all $\epsilon>0$, for all balls $B$, 
  $$
  (\mbox{Osc}_X (B))^q \lesssim \mbox{Osc}_V (B)^{(H-\epsilon) q}
  $$
  Hence,  it is not hard to deduce that $\tau_{X^H}(q)\ge \tau_{V}(Hq)$. In particular, 
  $$
  \tau_{X^H} (1)\ge \tau_V(H)
  $$
Let $I_{n,k}$ be the dyadic intervals $[k2^{-n}, (k+1)2^{-n})$ The graph of $X$ on $I_{n,k}$ can be covered by at most $\lfloor\frac{\mbox{Osc}_X(I_{n,k})}{2^{-n}}\rfloor+1$ boxes of side length $2^{-n}$. Therefore, 
$$
\overline{\mbox{dim}}_B ({\mathcal G}(X, [0,1])) \le \limsup_{n\to\infty} \frac{\log \left(\sum_{k=0}^{2^n-1} (\frac{\mbox{Osc}_X(I_{n,k})}{2^{-n}} +1)\right)}{-\log 2^{-n}} \le 1-\tau_X(1) \le 1-\tau_V(H).
$$
The lemma follows.
\end{proof}

\medskip

Lemma \ref{lemma_tau_upper_bound} provides us a natural upper bound of the process in terms of the $L^q$ spectrum. The natural question will be to study if the equality can be attained, at least for natural classes of increasing functions. The answer turns out to be affirmative.

\medskip

 \subsection{Self-similar measures.} Let $0<r_i<1$ and let $0=d_0 \le d_1<...<d_{m-1} = 1-r_{m-1}$ be positive real numbers. We define 
$$
S_i (x) = r_ix+ d_i
$$
Then $\Phi = \{S_i: i = 0,...,m-1 \}$ defines an iterated function system (IFS) with a unique compact attractor $K$ satisfying 
$$
K  = \bigcup_{i=0}^{m-1} S_i(K)
$$
and for a given set of probability vectors ${\bf p} = (p_0,..,p_{m-1})$ where $p_0+...+p_{m-1} = 1$, a unique self-similar measure $\mu = \mu (\Phi, {\bf p})$ is defined as the measure  satisfying 
$$
\mu(E) = \sum_{i=0}^{m-1} p_i \mu (S_i^{-1}(E)), \ \forall E \ \mbox{Borel}. 
$$
We say that the IFS satisfies the {\it convex open set condition} if $S_i(0,1)\cap S_j(0,1) = \emptyset$ for all $i\ne j$. By our construction, the self-similar set $K\subset [0,1]$ and  $1\in K$ and $S_i(K)\cap S_j(K)$ intersects at most only one point. The self-similar measure $\mu$ is compactly supported inside $K$ and it  defines a continuous increasing functions $V (x) = \mu [0,x]$. The $L^q$ spectrum $\tau_V(q)$ of the above IFS has been well-studied in literature \cite{Falconer2013} and it is well-known that it satisfies the equation 
\begin{equation}\label{eq_multi-fractal}
\sum_{i=0}^{m-1} p_i^q r_i^{-\tau_V(q)} = 1.    
\end{equation}
It is well-known that $\tau_V$ is concave and increasing. For some basic knowledge about self-similar measures and this $L^q$ spectrum, we refer readers to \cite{Falconer2013, BP}.  Using the multiindex notation, $\Sigma^0 = \emptyset$, $\Sigma^k = \{0,1,...m-1\}^k$ for all $k\ge 1$ and $\Sigma^{\ast} = \bigcup_{k\ge 0}\Sigma^k$. If $\sigma\in\Sigma$ If $\sigma  = (\sigma_1,...,\sigma_k)\in \Sigma^k$, then we define $\sigma^{-} = (\sigma_1,...,\sigma_{k-1})$.  $\sigma\sigma'$ means the concatenation of two words $\sigma$ and $\sigma'$. 

\medskip

We can then project the multi-indices onto the self-similar sets and the probability vectors. 
$$
S_{\sigma} (x) = S_{\sigma_1}\circ...\circ S_{\sigma_k}(x).
$$
We let $K_{\sigma} = S_{\sigma}(K)$. We will use (\ref{eq_multi-fractal}) to define a new probability vector 
\begin{equation}\label{eq_q}
{\bf q} = (q_0,...,q_{m-1}), \ q_i = p_i^q r_i^{-\tau_V(q)}.    
\end{equation}
Let also $q_{\sigma} = q_{\sigma_1}....q_{\sigma_k}$ and similarly for $p_{\sigma}$ and $r_{\sigma}$. The measure $\nu$ is the unique self-similar measure
$$
\nu = \mu (\Phi,{\bf q}).
$$
\medskip

We also need to collect all those $K_{\sigma}$ with approximately equal diameter. Let $0<t<1$, we define 
$$
\Lambda_n = \{\sigma\in\Sigma^{\ast}: r_{\sigma} \le t^n < r_{\sigma^-}\}
$$
It is well-known that 
$$
K = \bigcup_{\sigma\in\Lambda_n}K_{\sigma}.
$$
Suppose that $\sigma\in\Lambda_n$, we define 
$$
\Lambda_{\sigma,n} = \left\{\sigma': \sigma\sigma'\in\Lambda_n\right\}.
$$
As a direct observation, 
$$
\Lambda_{n+1} = \bigcup_{\sigma\in\Lambda_n} \{\sigma\}\times \Lambda_{\sigma,n}.
$$
We should notice that for the equicontractive IFS (i.e. all $r_i = r$), we can let $t = r$ and then $\Lambda_n = \Sigma^n$ for all $n\in{\mathbb N}$. 

\begin{lemma}\label{lemma_finite}
There exists $C>0$ such that the cardinality of $\Lambda_{\sigma,n}$, denoted by $\#\Lambda_{\sigma,n}$, is bounded above by $C$ uniformly for all $\sigma\in\Lambda_n$ and $n\in{\mathbb N}$. 
\end{lemma}

\begin{proof}
    Let $n\in{\mathbb N}$ and let $\sigma\in\Lambda_n$. Note that if $\sigma'\in\Lambda_{\sigma,n},$ then 
    $$
    r_{\sigma}r_{\sigma'} \le t^{n+1}<r_{\sigma}r_{\sigma'^{-}} \ \mbox{and} \  \ r_{\sigma}\le t^n< r_{\sigma^{-}}.
    $$
    From here, we can deduce that 
    $$
    r_{\min}t \le r_{\sigma'}\le r_{\min}^{-1} t
    $$
    where $r_{\min} = \min\{r_0, r_1,...,r_{m-1}\}$. Hence, for all $\sigma$ and for all $n\in{\mathbb N}, $
    $$
    \Lambda_{\sigma, n} \subset \{\sigma': r_{\min}t \le r_{\sigma'}\le r_{\min}^{-1} t\}
    $$
    The later has only finitely many elements, so the lemma follows. 
\end{proof}
    
\begin{lemma}\label{lemma_Lq}
Suppose that the IFS satisfies the convex open set condition and let $\mu$ be the associated self-similar measure with $V(x) = \mu[0,x]$. Then for all $q>0$, 
$$
\tau_{V}(q) = \liminf_{n\to\infty} \frac{\log\left(\sum_{\sigma\in\Lambda_n} (\mu (S_{\sigma}[0,1]))^q\right)}{n\log t}
$$
\end{lemma}
\begin{proof}
    Denote for the moment the right hand side of the claimed equality by $T_V(q)$. As $\{(S_{\sigma}(0,1))\}_{\sigma\in\Lambda_n}$ is one of the collections of disjoint balls, $\tau_V(q)\le T_V(q)$. Conversely, the open set condition implies that there exists $D>0$ such that  if $B$ is an open ball of radius $r$ and $r\in[t^{n+1},t^{n})$, 
    $$
    \#\{\sigma\in\Lambda_n: B\cap S_{\sigma}[0,1]\ne\emptyset\}\le D.
    $$
    This means that for any collection of disjoint balls ${\mathcal B}$ of radius $r$, 
    $$
    \begin{aligned}
    \sum_{B\in{\mathcal B}} (\mu (B))^q\le  &   \sum_{B\in{\mathcal B}} \left(\sum_{\sigma: S_{\sigma}[0,1]\cap B\ne \emptyset} \mu (S_{\sigma}[0,1])\right)^q \\
    \le & \sum_{B\in{\mathcal B}} D^q \sum_{\sigma: S_{\sigma}(0,1)\cap B\ne \emptyset} (\mu (S_{\sigma}(0,1)))^q \\
    \le & 2D^q\sum_{\sigma\in\Lambda_n}(\mu (S_{\sigma}[0,1]))^q
    \end{aligned}
    $$
    where we used  the inequality $(x_1+...+x_D)^q\le D^q(x_1^q+...+x_D^q)$ for all $q>0, x_i\ge 0$ in the second line and  the fact that each $B$ intersects at most two $S_{\sigma}[0,1]$ for all $\sigma\in\Lambda_n$ with positive measure in the third line. This inequality immediately implies that $T_V(q)\le \tau_V(q)$. 
\end{proof}

\section{Proof of Theorem \ref{theorem_multifractal}.} We are now ready to prove Theorem \ref{theorem_multifractal} stated in the introduction. As we noticed in the previous section, we will use $\tau_V(q)$ and $\tau_{\mu}(q)$ interchangably as they are equal on $q>0$.

\medskip

We  first begin by noting that $ \nu\times \nu$ is the self-affine measure supported on $K\times K$ satisfying  
\begin{equation}\label{eqnu}
\nu \times \nu = \sum_{i,j=0}^{m-1} q_iq_j (\nu\times \nu) \circ T_{i,j}^{-1},
\end{equation}
where $T_{i,j}(x,y) = (S_i(x),S_j(y))$. 

\medskip

Let $f: K\times K$ be a non-negative function and $f(t,u) = f(u,t)$. Observe that 
$$
\begin{aligned}
K\times K  =& \left( \bigcup_{\sigma\ne\sigma'\in\Lambda_1} (K_{\sigma}\times K_{\sigma'})\cup (K_\sigma\times K_{\sigma'})\right) \cup \bigcup_{\sigma\in\Lambda_1}(K_\sigma\times K_\sigma)\\
= & \bigcup_{k=0}^{1} \bigcup_{\sigma\in \Lambda_k}\bigcup_{\sigma'\ne\sigma''\in\Lambda_{\sigma,k}}(K_{\sigma  \sigma'}\times K_{\sigma  \sigma''} \cup K_{\sigma  \sigma'}\times K_{\sigma \sigma''}) \cup \bigcup_{\sigma\in\Lambda_2}(K_\sigma\times K_\sigma)\\
 = & \ \vdots\\
 = & \bigcup_{k=0}^{\infty} \bigcup_{\sigma\in \Lambda_k}\bigcup_{\sigma'\ne\sigma''\in\Lambda_{\sigma,k}}(K_{\sigma  \sigma'}\times K_{\sigma  \sigma''} \cup K_{\sigma  \sigma'}\times K_{\sigma \sigma''}) 
\end{aligned}
$$
Heuristically, we decompose $K\times K$ into smaller pieces by its distance around $t^k$ from the diagonal line $y = x$.

\medskip

As $f$ has the same values on $K_{\sigma1}\times K_{\sigma2}$ and $K_{\sigma2}\times K_{\sigma1}$, we have the following equation 
$$
\int\int_{K\times K} f(t,u)d\nu(t)d\nu(u) = 2\sum_{k=0}^{\infty} \sum_{\sigma\in \Lambda_k}\sum_{\sigma'\ne\sigma''\in\Lambda_{\sigma,k} } \int\int_{K_{\sigma i}\times K_{\sigma j}} f(t,u)d\nu(t)d\nu(u).
$$
Applying the self-similar identity (\ref{eqnu}) and with open set condition, 
$$
\int\int_{K_{\sigma \sigma'}\times K_{\sigma \sigma''}} f(t,u)d\nu(t)d\nu(u) = q_{\sigma}^2 \int\int_{K_{\sigma'}\times K_{\sigma''}} f(S_{\sigma}(t),S_{\sigma}(u)) d\nu(t)d\nu(u).
$$
 Therefore, we have 
\begin{equation}\label{eqE}
\int\int_{{K\times K}} f(t,u)d\nu(t)d\nu(u) =2\sum_{k=0}^{\infty} \sum_{\sigma\in \Lambda_k}\sum_{\sigma'\ne\sigma''\in\Lambda_{\sigma,k} } q_{\sigma}^2  \int\int_{K_{\sigma'}\times K_{\sigma''
}} f(S_{\sigma}(t),S_{\sigma}(u)) d\nu(t)d\nu(u).
\end{equation}

\medskip

Let $X(t) = B^H_{V(t)}$. We consider the graph measure 
$$
\nu_{\mathcal G} (E) = \nu \{t\in K: (t,X(t))\in E\}.
$$
The $s$-energy of $\nu_{\mathcal G}$, $1<s<2$, is given by 
$$
I_s(\nu_{\mathcal G}) = \int\int_{K\times K} \frac{1}{\left((t-u)^2+ (X(t)-X(u))^2\right)^{s/2}}d\nu(t)d\nu(u).
$$
Using Fubini's theorem, we have 
$$
{\mathbb E} [I_s(\nu_{\mathcal G})] = \int\int_{K\times K} f(t,u)d\nu(t)d\nu(u)
$$
where 
$$
\begin{aligned}
f(t,u) = &{\mathbb E}\left[ \frac{1}{\left((t-u)^2+(B^H(V(t))-B^H(V(u)))\right)^{s/2}}\right]\\
 = & \frac{1}{\sqrt{2\pi}}\int_0^{\infty} \frac{1}{\left((t-u)^2+(V(t)-V(u))^{2H}x^2\right)^{s/2}} e^{-x^2/2}dx.
\end{aligned}
$$

\medskip

The convex open set condition implies that $S_i[0,1]$ disjoint intervals or intersect at most one point.  Combining with the definition of $V(t) = \mu [0,t]$, it follows that 
$$
V(S_i(x))  = p_0+...+p_{i-1} + p_i V(x).
$$
for all $i = 1,..., m-1$  and $V(S_0(x)) = p_0 V(x)$. Moreover, 
$$
V(S_{\sigma}(x)) =p_{\sigma}V(x)+ t_{\sigma}, \ t_{\sigma} = V(S_{\sigma}(0))
$$
(the exact value of $t_{\sigma}$ is  not an important number to us). 
This then gives that 
\begin{equation}\label{eq_f}
f(S_{\sigma}(t),S_{\sigma}(u)) =  \frac{1}{\sqrt{2\pi}}\int_0^{\infty} \frac{1}{\left[\left(r_{\sigma}(t-u)\right)^2+\left(p_{\sigma}^{2H}(V(t)-V(u))^{2H}\right)x^2\right]^{s/2}} e^{-x^2/2}dx.
\end{equation}
 With a similar approach as in Lemma \ref{boundonexpec}, we let 
$$
c = \frac{ r_{\sigma} (t-u)}{p_{\sigma}^H (V(t)-V(u))^{H}}. 
$$
 We decompose the integrand in (\ref{eq_f}) into $\int_0^c$ and $\int_c^{\infty}$, and then estimate  
$$
\begin{aligned}
f(S_{\sigma}(t),S_{\sigma}(u)) \lesssim & \int_0^c \frac{1}{r_{\sigma}^{s}(t-u)^s}dx + \int_c^{\infty} \frac{1}{p_{\sigma}^{Hs}(V(t)-V(u))^{Hs}x^s}dx\\
\lesssim & \ \frac{c}{r_{\sigma}^{s}(t-u)^s} + \frac{1}{p_{\sigma}^{Hs}(V(t)-V(u))^{Hs}} c^{1-s}\\
 \lesssim & \  \frac{1}{r_{\sigma}^{s-1}(t-u)^{s-1}p_{\sigma}^{H} (V(t)-V(u))^H}.  \ (\mbox{by plugging in the value of c}) . 
\end{aligned}
$$
\medskip
We now plug in this estimate into (\ref{eqE}) and this yields,
$$
\begin{aligned}
\int\int_{K\times K} f(t,u)d\nu(t)d\nu(u)  \lesssim& \sum_{k=1}^{\infty}\sum_{\sigma\in \Lambda_k} \frac{q_{\sigma}^2}{r_{\sigma}^{s-1}p_{\sigma}^{H}} \\
 &  \ \ \cdot  \left(\sum_{\sigma'\ne\sigma''\in\Lambda_{\sigma,k} } \int\int_{K_{\sigma'}\times K_{\sigma''}}  \frac{1}{(t-u)^{s-1} (V(t)-V(u))^{H}} d\nu(t)d\nu(u)\right)\\
\end{aligned}
$$

\medskip

We will prove that  the inside sum is finite over all $\sigma$ and $k$ in the following lemma.

\begin{lemma}\label{lemma_finite_integral}Suppose that $\nu$ is the measure associated with the probability vector in (\ref{eq_q}) with $1>q\ge H$ and $s< 1-\tau_V(q)$. Then
$$
\sup_{n\ge 1}\sup_{\sigma\in\Lambda_n} \left(\sum_{\sigma'\ne\sigma''\in\Lambda_{\sigma,k}}\int\int_{K_{\sigma'}\times K_{\sigma''}}  \frac{1}{(t-u)^{s-1} (V(t)-V(u))^{H}} d\nu(t)d\nu(u)\right)<\infty.
$$
\end{lemma}

\medskip

Assuming for the moment that  this lemma is proven, we now investigate the whole summation.  Invoking the definition of $q_i$ in (\ref{eq_q}) and assuming $q\ge H$ so that the above lemma is valid, we have that 
\begin{equation}\label{eq_sigma}
\sum_{\sigma\in \Lambda_k} \frac{q_{\sigma}^2}{r_{\sigma}^{s-1}p_{\sigma}^{H}} = \sum_{\sigma\in \Lambda_k} \frac{p_{\sigma}^{2q} r_{\sigma}^{-2\tau_V(q)}}{r_{\sigma}^{s-1}p_{\sigma}^{H}} = \frac{1}{r_{\sigma}^{s-1+2\tau_V(q)}} \sum_{\sigma\in\Lambda_k} p_{\sigma}^{2q-H}.    
\end{equation}
Using the definition of $\tau_V$ with Lemma \ref{lemma_Lq} and $q>H/2$,
we have for all $\epsilon>0$ and for all $k\ge k_{\epsilon}$,
$$
\sum_{I\in\Lambda_k}p_{\sigma}^{2q-H} \le t^{k(\tau_V(2q-H)-\epsilon)}. 
$$
Putting back to (\ref{eq_sigma}) and $r_{\sigma}\ge r_{\min}t^k$ for all $\sigma\in \Lambda_k$, we obtain that 
$$
{\mathbb E}[I_s(\nu_{\mathcal G})]\lesssim \sum_{k\ge k_{\epsilon}}\sum_{\sigma\in \Sigma^k} \frac{q_{\sigma}^2}{r_{\sigma}^{s-1}p_{\sigma}^{H}} \lesssim \sum_{k\ge k_{\epsilon}}\frac{1}{t^{k[s-1+2\tau_V(q)-\tau_V(2q-H)+\epsilon]}}.
$$
Since $t<1$, this sum is finite if 
$$
s-1+2\tau(q)-\tau_V(2q-H)+\epsilon<0, 
$$
i.e. $s< 1-2\tau_V(q)+\tau_V(2q-H)-\epsilon$. If we take $q = H$, we have for all $s<1-\tau_V(H)-\epsilon$ and  
$$
{\mathbb E}[I_s(\nu_{\mathcal G})] <\infty.
$$
 This shows that $\mbox{dim}_H({\mathcal G}(X,K))\ge 1-\tau_V(H)$ by letting $\epsilon\to 0$ and $s$ approaches $1-\tau_V(H)$. 
 Since ${\mathcal G}(X,[0,1]) = {\mathcal G}(X,K)\cup {\mathcal G}(X,[0,1]\setminus K)$ and $X$ is almost surely constant on $[0,1]\setminus K$.  The later has Hausdorff dimension 1. By countable stability, dim$_H \ {\mathcal G}(X,[0,1])=\mbox{dim}_H{\mathcal G}(X,K)\ge   1-\tau_V(H)$.  This completes the proof. \qquad$\Box$

 \medskip
 
 It remains to prove Lemma \ref{lemma_finite_integral}.
 \medskip
 
 \noindent{\it Proof of Lemma \ref{lemma_finite_integral}.} We first note that from Lemma \ref{lemma_finite}, it suffices to show that  for all $\sigma\ne \sigma'\in \{\sigma: r_{\min }t \le r_{\sigma}\le r_{\min}^{-1}t\}$, the integral
 $$
 \int\int_{K_{\sigma}\times K_{\sigma'}}  \frac{1}{(t-u)^{s-1}(V(t)-V(u))^{H}} d\nu(t)d\nu(u)<\infty.
 $$
 Without loss of generality, we can assume $S_{\sigma}(1)\le S_{\sigma'}(0)$ so that the interval $S_{\sigma}[0,1]$ (or equivalently $K_{\sigma}$) is in the left hand side of $S_{\sigma'}[0,1]$ (or $K_{\sigma'}$). Let $n$ be the first integer such that $t^n<r_{\min }t$. We now decompose $S_{\sigma}[0,1]$ inductively in the following ways:
 $$
 \Xi_{n} = \left\{\eta_n: \sigma\eta_n\in \Lambda_{n+1} \right\}
 $$
As $K_{\sigma\eta_n}$ intersects at most one point, we let $K_{\sigma\eta_n^{\ast}}$ be the rightmost piece of  $K_{\sigma\eta_n}$ in $K_{\sigma}$. 
Then 
$$
K_{\sigma} = J_n\cup K_{\sigma\eta_n^{\ast}} \  \ \mbox{where}  \ \ J_n = \bigcup_{\eta_n\in \Xi_n\setminus\{\eta_n^{\ast}\}} K_{\sigma\eta_n}.
$$
Inductively, we define
 $$
 \Xi_{k} = \left\{\eta_k : \sigma\eta_n^{\ast}\eta_{n+1}^{\ast}....\eta_{k-1}^{\ast}\eta_{k}\in \Lambda_{k+1} \right\}, \ k\ge n
 $$
 Also, $ K_{\sigma\eta_n...\eta_{k}^{\ast}}$ be the rightmost piece in $K_{\sigma\eta_n^{\ast}....\eta_{k-1}^{\ast}}$ so that 
 $$
K_{\sigma\eta_n^{\ast}...\eta_{k-1}^{\ast}} = J_k\cup  K_{\sigma\eta_n...\eta_{k}^{\ast}} \ \mbox{and} \  J_k =\bigcup_{\eta_k\in \Xi_k\setminus\{\eta_k^{\ast}\}} K_{\sigma\eta_n^{\ast}...\eta_{k-1}^{\ast}\eta_k}.
 $$
 Therefore, we have 
 $$
 K_{\sigma} = \bigcup_{k\ge n} J_k.
 $$
We now notice that it $t\in J_k$ and $u\in K_{\sigma'}$, the set $K_{\sigma\eta_n^{\ast}....\eta_{k}^{\ast}}$ is always in between $J_k$ and $K_{\sigma}'$. Hence, it holds that for all $t\in J_k$ and $u\in K_j$, 
 $$
|t-u| \ge r_{\sigma}r_{\eta_n^{\ast}}...r_{\eta_k^{\ast}} \ \mbox{and} \   |V(t)-V(u)|\ge p_{\sigma}p_{\eta_n^{\ast}}...p_{\eta_k^{\ast}}.
 $$
 Notice that 
 $$
( \nu\times\nu)(J_k\times K_{\sigma}')\le \nu (J_k) = p_{\sigma}^qp_{\eta_1^{\ast}}^q...p_{\eta_{k-1}^{\ast}}^qr_{\sigma}^{-\tau_V(q)}r^{-\tau_V(q)}_{\eta_n^{\ast}}...r_{\eta_{k-1}^{\ast}}^{-\tau_V(q)} \left( \sum_{\eta_k\in\Xi_k\setminus\{\eta_{k}^{\ast}\}}p_{\eta_k}^qr_{\eta_k}^{-\tau_V(q)}\right).
 $$
 Hence, combining all these estimate and decomposing $K_{\sigma}\times K_{\sigma'}$ into unions of $J_k\times K_{\sigma}'$, we have 
 $$
 \begin{aligned}
  &\int\int_{K_{\sigma}\times K_{\sigma'}}  \frac{1}{(t-u)^{s-1}(V(t)-V(u))^{H}} d\nu(t)d\nu(u)\\ \le  &\sum_{k\ge n}   (r_{\sigma}r_{\eta_n^{\ast}}...r_{\eta_k^{\ast}})^{-(s-1)}(p_{\sigma}p_{\eta_n^{\ast}}...p_{\eta_k^{\ast}})^{-H}  (\nu\times \nu) (J_k\times K_{\sigma'})\\
   \le &  \sum_{k\ge n}   
   (r_{\sigma}r_{\eta_n^{\ast}}...r_{\eta_k^{\ast}})^{-(s-1)}(p_{\sigma}p_{\eta_n^{\ast}}...p_{\eta_k^{\ast}})^{-H} p_{\sigma}^qp_{\eta_1^{\ast}}^q...p_{\eta_{k-1}^{\ast}}^qr_{\sigma}^{-\tau_V(q)}r^{-\tau_V(q)}_{\eta_n^{\ast}}...r_{\eta_{k-1}^{\ast}}^{-\tau_V(q)} \left( \sum_{\eta_k\in\Xi_k\setminus\{\eta_{k}^{\ast}\}}p_{\eta_k}^qr_{\eta_k}^{-\tau_V(q)}\right)  \\
   \le  & \sum_{k\ge n}   
   (r_{\sigma}r_{\eta_n^{\ast}}...r_{\eta_{k-1}^{\ast}})^{-(s-1)-\tau_V(q)} \cdot p_{\eta_k^{\ast}}^{-H} \left( \sum_{\eta_k\in\Xi_k\setminus\{\eta_{k}^{\ast}\}}p_{\eta_k}^qr_{\eta_k}^{-\tau_V(q)}\right)
 \end{aligned}
 $$
   where in the last line we used $q\ge H$ and  $p_i<1$.  From the definition of $\Xi_k$ and Lemma \ref{lemma_finite}, there is a uniform upper bound $C$ on $\#\Xi_k$ for all $k$. Moreover, the length of $\eta\in\bigcup_{k}\Xi_k$ is uniformly bounded by certain integer $L$.  Hence, 
  $$
  p_{\eta_k^{\ast}}^{-H} \left( \sum_{\eta_k\in\Xi_k\setminus\{\eta_{k}^{\ast}\}}p_{\eta_k}^qr_{\eta_k}^{-\tau_V(q)}\right)\le \frac{r_{\max}^{-\tau_V(q)}}{p_{\min}^{LH}} \cdot\#\Xi_k\le \frac{Cr_{\max}^{-\tau_V(q)}}{p_{\min}^{LH}}. 
  $$
   Hence, 
   $$
   \begin{aligned}
  \int\int_{K_{\sigma}\times K_{\sigma'}}  \frac{1}{(t-u)^{s-1}(V(t)-V(u))^{H}} d\nu(t)d\nu(u) \lesssim &\sum_{k\ge n}   
   (r_{\sigma}r_{\eta_n^{\ast}}...r_{\eta_{k-1}^{\ast}})^{-(s-1)-\tau_V(q)} \\
   \lesssim & \sum_{k\ge n}   t^{k (1-s-\tau_V(q))}.    
   \end{aligned}
     $$
This geometric series is finite as long as $s<1-\tau_V(q)$. The proof is complete. \qquad$\Box$

\section{Fourier dimension for Brownian Staircase}
In this section, we will prove Theorem \ref{thm_BStair}. The Hausdorff dimension has been computed in the previous section, we just need to study its Fourier dimension. 
 We will show that the Fourier dimension of the Brownian staricase is zero. Our idea is to make use of countable stability. Indeed, there is no countable stability of Fourier dimension in general \cite{E}, however, to overcome this,  Ekstr\"{o}m, Persson and Schmeling \cite{EPS2015} defined the modified Fourier dimension 
$$
{\rm dim}_{FM} A = \sup \{\mbox{dim}_F(\mu): \mu(A)>0, \  \mu \ \mbox{is a  finite Borel  measure}\}
$$
where $\mbox{dim}_F(\mu)= \sup \{s\ge 0: |\widehat{\mu}(\xi)| \lesssim |\xi|^{-s/2}\}$. Notice that ${\rm dim}_{F} A\le {\rm dim}_{FM} A$. They proved the following theorem giving a sufficient condition for countable stability of Fourier dimension.

\begin{theorem}\cite[Corollary 3]{EPS2015}\label{thm_EPS}
Let $A_k$, $k=1,2,3...$ be a countable family of Borel sets on $\R^d$ such that 
\begin{equation}\label{eq_conditionFM}
\sup_{n} \mbox{\rm dim}_{FM} \left( A_n \cap \overline{\bigcup_{k\ne n} A_k}\right) \leq \sup_k \mbox{\rm dim}_F A_k.
\end{equation}
Then 
$$
dim_F \left(\bigcup_{k=1}^{\infty} A_k\right) = \sup_{k}\mbox{\rm dim}_F A_k.
$$
\end{theorem}

 Let $K$ be the middle-third Cantor set. It is well-known that ${\rm dim}_F K= 0$ and all measures supported on $K$ has no Fourier decay \cite[Theorem 8.1]{Ma}. It is also known that ${\rm dim}_{FM} K = 0$. As mentioned in \cite{EPS2015,JS2016},  it can be derived from Davenport, Erd\"{o}s and LeVeque's \cite{DEL1963} equidistribution result that if $\widehat{\mu}$ decays polynomially, then $\mu$ almost everywhere $x$ is normal to any bases. (i.e. the fractional part $\{b^kx\}$ is uniformly distributed on [0,1] for all $b = 2,3,...$). However, since all numbers in the middle-third Cantor set $K$ are not normal to the base 3 (as the digit 1 is always omitted),  this implies that none of the measure $\mu$ which has a positive measure on $K$ has polynomial decay. In particular, ${\rm dim}_{FM} K = 0$.  This allows us to show the following lemma.  

\begin{lemma}\label{lemmaFdimgraph}
Let $K$ be the middle-third Cantor set and let  $X: K\to \R$ be any continuous functions supported on $K$. Then ${\mathcal G}(X,K)$ does not support any measure whose Fourier transform decays. Moreover,  
$$
{\rm dim}_F \ {\mathcal G}(X,K) = 0 \ \mbox{and} \   {\rm dim}_{FM} \ {\mathcal G}(X,K) = 0.
$$  
\end{lemma}

\begin{proof}
Let $\mu$ be any Borel measures supported on ${\mathcal G}(X,K)$. Consider the projection map $P: {\mathcal G}(X,K)\to K$ defined by $P (t,y) = t$ (and $y= X(t)$). Note that $P$ is a bijective function. Define now the measure $\nu$ on $K$ defined by 
$$
\nu (E) = \mu (P^{-1}(E)).
$$
We claim that $\mu (F) = \nu\{t\in K: (t,X(t))\in F\}$, which means that  all measures on the graph must be a push-forward of some measure from the domain. To see this, for any Borel measurable function $f$ supported on the graph,  
$$
\int f(t,X(t))d\nu(t) = \int f(P(t,y),X(P(t,y))) d\mu(t,y) = \int f(t, y) d\mu(t,y).
$$
since $X(P(t,y)) = X(t) = y$. This justifies the claim.

\medskip

Now, taking Fourier transform of $\mu$ we have
$$
\widehat{\mu}(\xi_1,\xi_2) = \int_K e^{-2\pi i (\xi_1 t+\xi_2. X(t))}d\nu(t).
$$
Consider $\xi_2 = 0$. We have that $\widehat{\mu}(\xi_1,0) = \widehat{\nu}(\xi_1)$. As we know $\nu$ is a Borel measure on the middle third Cantor set, the measure does not decay. Hence, $\mu$ does not decay on the $x$-axis.

\medskip

To show that ${\rm dim}_{FM} \ {\mathcal G}(X,K) = 0$. We let $\mu$ be a Borel measure such that $\mu ({\mathcal G}(X,K))>0$.  Let $Q:\R^2\to\R$ be the orthogonal projection map onto the $x$-axis. Consider $\nu(E) = \mu (Q^{-1}(E))$. Then 
$$
\nu (K) = \mu (Q^{-1}(K)) \ge \mu ({\mathcal G}(X,K))>0
$$
since $Q^{-1}(K)$ contains ${\mathcal G}(X,K)$. Therefore, dim$_F\nu = 0$. Notice that $\widehat{\mu}(\xi,0) = \widehat{\nu}(\xi)$. Hence, $\mbox{dim}_F\mu = 0$ also. This implies that  ${\rm dim}_{FM} \ {\mathcal G}(X,K) = 0.$
\end{proof}

\medskip

\noindent{\it Proof of Theorem \ref{thm_BStair}.} We have already shown the Hausdorff dimension estimate in the previous subsection. In order to show that the Brownian staircase function $X_t = B_{V(t)}$ has Fourier dimension zero,  we decompose 
$$
[0,1]\setminus K = \bigcup_{n=1}^{\infty} I_n.
$$
where $I_n$ are open disjoint intervals. Then 
$$
{\mathcal G} (X, [0,1]) = {\mathcal G} (X, K) \cup \bigcup_{n=1}^{\infty} {\mathcal G} (X, I_n).
$$ 
We have shown that ${\rm dim}_F{\mathcal G} (X, K) = 0.$ As $X_t$ is a constant function on each $I_n$, so ${\rm dim}_F{\mathcal G} (X, I_n) = 0$ also. Our theorem will follow from Theorem \ref{thm_EPS} if we can verify (\ref{eq_conditionFM}). Let $A = {\mathcal G} (X, K)$ and let $A_n = {\mathcal G} (X, I_n)$. We  notice that 
$$
A\cap\overline{ \bigcup_{n=1}^{\infty} A_n} = {\mathcal G} (X, K) \cap \overline{\bigcup_{n=1}^{\infty}{\mathcal G} (X, I_n)} = {\mathcal G} (X, K)\cap {\mathcal G} (X, [0,1]) = {\mathcal G} (X, K)
$$
By Lemma \ref{lemmaFdimgraph},  it has modified Fourier dimension zero. And
$$
A_k\cap\overline{ \bigcup_{n\ne k}^{\infty} A_n \cup A} =  {\mathcal G} (X, I_k) \cap  {\mathcal G} (X, [0,1]\setminus I_k) = \emptyset
$$
 which has modified Fourier dimension zero also.  This completes the proof. \qquad$\Box$

 \medskip

 \section{Remark and open questions}
 
Let us conclude the paper with some remarks and open questions. 

\medskip

(1). We mainly deal with centered continuous additive processes. Indeed, if $X_t$ is a (non-centered) continuous additive processes,  then $X_t -\E[X_t]$ is a centered continuous additive processes. Hence, any continuous additive processes can be written as 
 $$
 X_t = f(t)+B_{V(t)}
 $$
 where $f$ is some deterministic continuous function and $B_{V(t)}$ is the centered continuous additive process. For regular functions (e.g. bi-Lipschitz) $f$, the Hausdorff dimension will not change. However, for more related questions concerning the relationship of the graph of $B_t+f$ and $B_t$, readers may refer to \cite{PS2016}. 
 
 \bigskip
 
(2). It was mentioned in \cite{Xiao2003}, any L\'{e}vy processes $X_t$ induce many additive processes by a time-change $X_{V(t)}$ where $V$ is a right-continuous increasing function. We do not know if these will characterize all additive processes, but it is likely that the dimensions of these types of additive processes can be studied through the Blumental-Getoor indices and the H\"{o}lder regularity of $V$. 
 
\bigskip

\noindent(3). For continuous singular strictly increasing function $V$ (e.g. Example \ref{example6.3}), we do not have a sharp Hausdorff dimension estimate. However, the $L^q$ spectrum and the multifractal formalism of  Bernoulli convolution associated with golden ratio are fully computed in \cite{Feng2005}. With a careful study, we believe that the Hausdorff dimension of the graph of $B^H_{V(t)}$ will be $1-\tau_V(H)$ for the $V$ in Example \ref{example6.3}. More generally, the following conjecture may be true:
 
 \begin{conj}
 Let $H\in(0,1)$ and let $X^H_t = B^H_{V(t)}$ Suppose that the $L^q$-spectrum of the increasing function $V$ is differentiable at $H$ and it satisfies  the multifractal formalism at $\alpha = \tau_V'(H).$ 
 Then
 $$
 \mbox{dim}_H {\mathcal G}(X^H, [0,1]) = 1-\tau_V(H). 
 $$
 \end{conj}
 
 \medskip
 
 The conjecture can be resolved if we can show that 
\begin{equation}\label{eqKalpha}
 \mbox{dim}_H  \ {\mathcal G}(X^ H, K_{\alpha}) = 1+\mbox{dim}_H(K_\alpha) -H\alpha.
   \end{equation}
 where $K_{\alpha}$ is the level set of local dimensions  at $\alpha = \tau'(H)$. Indeed, the assumption  of the conjecture  implies that 
 $$
 \mbox{dim}_H(K_\alpha) = \tau^{\ast}(\alpha) = H \tau'(H)-\tau(H). 
 $$
 Hence,  
 $$
  \mbox{dim}_H  \ {\mathcal G}(X^H, [0,1])\ge \mbox{dim}_H  \ {\mathcal G}(X^ H, K_{\alpha}) = 1-\tau(H).
 $$
As upper bound has been  shown to be  true always, this resolved the conjecture. 

\medskip

To resolve (\ref{eqKalpha}), we only need to establish the lower bound. However, the definition of local dimension  only implies  that for $r$ sufficiently  small depending on $t\in K_{\alpha}$, we have
$$
|V(t+r)-V(t-r)|\ge r^{\alpha+\epsilon},
$$
but to  apply (\ref{eq_bound_expect_frac}), we  need a uniform estimate that for sufficiently small $r$ depending on $t\in K_{\alpha}$
$$
|V(s)-V(u)|\ge |s-u|^{\alpha+\epsilon}, \ \forall  \ s,u\in (t-r,t+r).
$$
There  is a subtle difference between a local estimate  and a locally uniform estimate, refraining us  from obtaining the strongest conclusion. The problem is avoided  by a more careful  estimate using self-similarity in Theorem \ref{theorem_multifractal}
 \bigskip
 
  
 (4). Concerning the  Fourier dimension, Example \ref{example6.3} showed that there is a strictly increasing $V$ with the graph $B_{V(t)}$ supports a power Fourier decay despite the fact that the measure $\mu = dV$ has no Fourier decay. The following question are interesting to get further investigation. 
 
 \medskip
 
 {\bf (Qu 1).} Suppose that $V$ is strictly increasing. is it true that the graph of $B_{V(t)}$ always supports a Rajchman measure?

 \medskip
 
 Also, the sharp bound of  the Fourier dimension for bi-Lipschitz function is also  worth to investigate. 
 
 {\bf (Qu 2).} It is only known that the graph of the standard Brownian motion has Fourier dimension 1. In this paper, 
we showed that if $V$ is bi-Lipschitz, then the Fourier dimension of graph is at least 2/3. We do not know if 1 can be attained here. In view of this, what would be the exact almost sure Fourier dimension of graph of $B_{V(t)}$ if $V$ is bi-Lipschtiz?

\medskip

\noindent {\bf Acknowledgement.} The authors would like to thank the referees for carefully reading the manuscript and making many constructive comments. In particular, they would like thank the referees suggesting the papers \cite{E2,Jin2011} and encouraging them to generalize the proof of Cantor devil stair function to the form in Theorem \ref{theorem_multifractal}.

 \end{document}